\documentclass[a4paper, 9pt]{article}
\usepackage{graphics}
\usepackage[dvipdfmx]{hyperref}
\usepackage{color}

\usepackage{lscape}
\usepackage{latexsym}
\usepackage{amsmath,verbatim}
\usepackage{graphics}
\usepackage{amsthm}
\usepackage{amssymb}
\usepackage[mathscr]{euscript}
\usepackage{yfonts}
\usepackage{makeidx}
\usepackage{stmaryrd}
\usepackage{multicol}
\usepackage{bm}
\usepackage[all]{xy}
\numberwithin{equation}{section}
\newtheorem{thm}{Theorem}[section]
\newtheorem{prop}[thm]{Proposition}
\newtheorem{lem}[thm]{Lemma}
\newtheorem{cor}[thm]{Corollary}
{\bf}{\it}

\newtheorem{fthm}{Theorem}{\bf}{\it}
{\bf}{\it}
{\bf}{\it}
{\bf}{\it}
{\bf}{\it}

\theoremstyle{definition}
\newtheorem{defn}[thm]{Definition}

{\bf}{\rm}

\theoremstyle{remark}
\newtheorem{ex}[thm]{Example}
\newtheorem{rem}[thm]{Remark}
{\bf}{\it}

\newtheorem{definition and corollary}[thm]{Definition and Corollary}

\newtheorem{fex}[fthm]{Example}{\it}{\rm}

\newcommand{\al}{\alpha}
\newcommand{\af}{\mathrm{af}}

\newcommand{\sh}{\mathrm{sph}}

\newcommand{\C}{{\mathbb C}}

\newcommand{\cO}{{\mathcal O}}

\newcommand{\End}{\mathrm{End}}

\newcommand{\bI}{{\mathbf I}}
\newcommand{\bh}{{\mathbf h}}
\newcommand{\bW}{{\mathbb W}}

\newcommand{\ch}{\mathrm{ch}}
\newcommand{\gch}{\mathrm{gch}}

\newcommand{\la}{\lambda}
\newcommand{\La}{\Lambda}
\newcommand{\lo}{\mathrm{loc}}
\newcommand{\ra}{\mathrm{rat}}
\newcommand{\mt}{{\mathtt t}}

\newcommand{\Gr}{\mathrm{Gr}}

\newcommand{\g}{\mathfrak{g}}
\newcommand{\gb}{\mathfrak{b}}
\newcommand{\sB}{\mathscr B}
\newcommand{\sGB}{\mathscr{GB}}

\newcommand{\h}{\mathfrak{h}}

\newcommand{\bG}{\mathbf{G}}
\newcommand{\bH}{\mathbf{H}}
\newcommand{\bL}{\mathbf{L}}
\newcommand{\bv}{\mathbf{v}}

\newcommand{\tI}{\mathtt{I}}
\newcommand{\tJ}{\mathtt{J}}
\newcommand{\GW}{\mathtt{GW}}

\newcommand{\gp}{\mathfrak{p}}

\newcommand{\bO}{\mathbb{O}}

\renewcommand{\P}{\mathbb{P}}

\newcommand{\bQ}{\mathbf{Q}}

\newcommand{\sA}{\mathscr{A}}

\newcommand{\sS}{\mathscr{S}}
\newcommand{\sR}{\mathscr{R}}
\newcommand{\sQ}{\mathscr{Q}}
\newcommand{\sH}{\mathscr{H}}
\newcommand{\sX}{\mathscr{X}}

\newcommand{\Q}{\mathbb{Q}}
\newcommand{\R}{\mathbb{R}}
\newcommand{\bX}{\mathbb{X}}

\newcommand{\Z}{\mathbb{Z}}
\newcommand{\si}{\frac{\infty}{2}}
\newcommand{\Span}{\mbox{\rm Span}}

\newcommand{\Gm}{\mathbb G_m}

\newcommand\ringring[1]{%
  {
   \mathop{\kern0pt #1}\limits^{
     \vbox to-1.85ex{
       \kern-2ex 
       \hbox to 0pt{\hss\normalfont\kern.1em \r{}\kern-.45em \r{}\hss}%
       \vss 
     }
   }
  }
}

\title{Darboux coordinates on the BFM spaces}
\author{Syu \textsc{Kato}\footnote{Department of Mathematics, Kyoto University, Oiwake Kita-Shirakawa Sakyo Kyoto 606-8502 JAPAN \tt{E-mail:syuchan@math.kyoto-u.ac.jp}}}

\begin{document}
\maketitle

\begin{abstract}
Bezrukavnikov-Finkelberg-Mirkovi\'c [Compos. Math. {\bf 141} (2005)] identified the equivariant $K$-group of an affine Grassmannian, that we refer as (the coordinate ring of) a BFM space \'a l\`a Teleman [Proc. ICM Seoul (2014)], with a version of Toda lattice. We give a new system of generators and relations of a certain localization of this space, that can be seen as a version of its Darboux coordinate. This establishes a conjecture in Finkelberg-Tymbaliuk [Progress in Math. {\bf 300} (2019)] that relates the BFM space of a connected reductive algebraic group with those of Levi subgroups.
\end{abstract}

\section*{Introduction}
Let $G$ be a connected reductive algebraic group over $\C$. Let $B$ be a Borel subgroup of $G$ and let $H \subset B$ be its maximal torus. Let $\Gr_G$ denote the (thin) affine Grassmannian of $G$. The $G$-equivariant $K$-group $K_G ( \Gr_G )$ of $\Gr_G$ admits the structure of an algebra, and it is identified with the phase space of the relativistic Toda lattice in \cite{BFM05}. In particular, the space $K_G ( \Gr_G )$ carries a Poisson bracket. Braverman-Finkelberg-Nakajima \cite{Nak15,BFN18,BFN19} constructed a commutative algebra $\sA ( G, V )$ for each representation $V$ of $G$, whose spectrum is supposed to be a part of the space of vacua in the corresponding three-dimensional gauge theory. The space $\Gr_G$ played an essential r\^ole there, and we have a Poisson algebra embedding
\begin{equation}
\sA ( G, V ) \hookrightarrow \sA ( G, \{ 0 \} ) = K_G ( \Gr_G ).\label{KCoulomb}
\end{equation}
In addition, Teleman \cite{Tel19} gives a recipe to understand $\sA ( G, V )$ from $K_G ( \Gr_G )$.

Associated to $G$, we have its flag manifold $\sB$. In \cite{Kat18c,Kat18d}, we have constructed a ring morphism connecting $K _G ( \Gr_G )$ with the equivariant quantum $K$-group $qK_G ( \sB )$ of $\sB$ (\cite{Giv00,Lee04}):
\begin{equation}
K _G ( \Gr_G ) _\lo \cong qK_G ( \mathscr B )_\lo,\label{KPeterson}
\end{equation}
where the subscripts $``\lo"$ denote certain localizations, whose meaning {\it differs} in the both sides. This result, commonly referred to as the $K$-theoretic Peterson isomorphism (\cite{LLMS17}), also exhibits an aspect of the rich structures of $K _G ( \Gr_G )$.

Finkelberg-Tymbaliuk \cite{FT19} extensively studied $K _{\mathop{GL} ( n )} ( \Gr_{\mathop{GL} ( n )} )$ and deduced an algebra morphism
\begin{equation}
K _{\mathop{GL} ( n )} ( \Gr_{\mathop{GL} ( n )} ) \longrightarrow K _L ( \Gr_L )\label{F-hom}
\end{equation}
for a connected (standard) Levi subgroup $L \subset \mathop{GL} ( n )$. As this homomorphism is an incarnation of the coproduct structure of their shifted affine quantum groups (and also as they have similar homomorphisms for homologies \cite{FKPRW}), they led to conjecture that (\ref{F-hom}) exists for every connected reductive $G$ and also with the extra $\Gm$-action given by the loop rotation action.

The goal of this paper is to answer this conjecture affirmatively as:

\begin{fthm}[$\doteq$ Theorem \ref{main} + Corollary \ref{mcor}]\label{fmain}
For each connected reductive subgroup $H \subset L \subset G$, we have a chain of injective algebra homomorphisms:
$$K _{G \times \Gm} ( \Gr_G ) \hookrightarrow K _{L \times \Gm} ( \Gr_L ) \hookrightarrow K _{H \times \Gm} ( \Gr_H ).$$
\end{fthm}

Since the main portion of Theorem \ref{fmain} is the case of simple and simply connected $G$, we concentrate into this case in the rest of this introduction.

Here $K _{H \times \Gm} ( \Gr_H )$ is the (quantized) Heisenberg algebra, and hence this embedding can be seen to equip each $K _{L \times \Gm} ( \Gr_L )$ with its Darboux coordinate system. In addition, Corollary \ref{genKG} supplies its modification that describes a certain localization of the ring $K _{L \times \Gm} ( \Gr_L )$. This makes $K _{G \times \Gm} ( \Gr_G )$ into (the quantized phase space of) an integrable system called the relativistic Toda lattice, as described in Bezrukavnikov-Finkelberg-Mirkovi\'c \cite{BFM05}. In view of the homology version of (\ref{KPeterson}) discovered by Peterson \cite{Pet97}, it can be understood as the $K$-theoretic version of the fundamental presentation of (equivariant) quantum cohomology of flag varieties due to Givental-Kim \cite{GK95} and Kim \cite{Kim99}.

In the course of the proof of Theorem \ref{fmain}, we exhibit the non-commutative version of the main result in \cite{Kat18c}:

\begin{fthm}[$\doteq$ Corollary \ref{ncGr} and Theorem \ref{ncPet}]\label{fnc}
We have a commutative diagram, whose bottom arrow is an isomorphism of non-commutative rings:
$$
\xymatrix{
& K_{H \times \Gm} ( \bQ_G^{\ra} ) & \\
K_{H \times \Gm} ( \Gr_G ) _\lo \ar[rr] \ar@{^{(}->}[ru]^{\Phi} & & qK_{H \times \Gm} ( \sB )_\lo \ar@{_{(}->}[lu]_{\Psi}
},$$
where $\bQ_G^{\ra}$ is the semi-infinite flag manifold of $G$ $(\cite{Kat18d})$. Moreover, all of these morphisms respect Schubert bases. 
\end{fthm}

Our strategy to prove Theorem \ref{fmain} is as follows: We first refine some of the algebraic arguments in \cite{Kat18c} to prove Theorem \ref{fnc}. Then, we transplant the natural operations of $K _{G \times \Gm} ( \bQ_G^\ra )$ and give an algebra generator set of a suitable localization $K _{G \times \Gm} ( \Gr_G )_\lo$ of $K _{G \times \Gm} ( \Gr_G )$ in term of the Heisenberg action of $K _{H \times \Gm} ( \Gr_H )$. These boil down the proof of Theorem \ref{fmain} into a comparison of integral structures. For this comparison, we prove the ($\Gm$-equivariant version of the) following, best expressed in the language of quantum $K$-groups.

Let $\sB^L$ be the flag variety of $L$. Let $\bX^*$ be the weight lattice of $H$. Let $\{ \varpi_i \}_{i \in \tI}$ be the set of fundamental weights with respect to $H \subset B$. We have line bundles $\cO_{\sB} ( - \varpi_i )$ and $\cO_{\sB^L} ( - \varpi_i )$ on $\sB$ and $\sB^L$, respectively. Let $Q^{\vee}_+$ denote the nonnegative span of positive coroots of $G$, and let $Q^{\vee}_{L,+}$ denote the nonnegative span of positive coroots of $L$. We have a natural inclusion $Q^{\vee}_{L,+} \subset Q^{\vee}_+$. Let us employ the definition of quantum $K$-groups as:
$$qK_G ( \sB ) = K_G ( \sB ) \otimes \C [\![Q^{\vee}_+]\!] \hskip 5mm \text{and} \hskip 5mm qK_L ( \sB^L ) = K_L ( \sB^L ) \otimes \C [\![Q^{\vee}_{L,+}]\!],$$
where $\beta \in Q^{\vee}_+$ defines a formal variable $Q^{\beta} \in \C [\![Q^{\vee}_+]\!]$. These spaces are equipped with the commutative ring structures whose multiplications are denoted by $\star$. The multiplication $\star$ coincides with the usual multiplications rules of $K_G ( \sB )$ or $K_L ( \sB^L )$ by setting $Q^{\beta} = 0$ for all $\beta \neq 0$.

\begin{fthm}[$\doteq$ Theorem \ref{qK-surj}]\label{flsurj}
There exists a surjective morphism of rings
$$qK_G ( \sB ) \longrightarrow \!\!\!\!\! \rightarrow qK_L ( \sB^L )$$
obtained by setting $Q^{\beta} \equiv 0$ for $\beta \in Q^{\vee}_+ \setminus Q^{\vee}_{L,+}$. This morphism sends the quantum multiplication of $\cO_{\sB} ( - \varpi_i )$ to the quantum multiplication by $\cO_{\sB^L} ( - \varpi_i )$ for each $i \in \tI$.
\end{fthm}

We remark that the classical analogue of Theorem \ref{flsurj} is an isomorphism, sometimes referred to as the ``induction equivalence". We present a direct proof in the main body of this paper, that yields an interesting representation theoretic consequence (Corollary \ref{inclWeyl}), though it holds in much greater generality (Theorem \ref{qK-surj-gen}). Theorems \ref{flsurj} and \cite[Theorem A]{Kat19a} upgrade the key observations in Leoung-Li \cite{LL10} to the $K$-theoretic settings.

\begin{fex}\label{exToda}
Assume that $G = \mathop{SL} ( n, \C )$. Let us choose the fundamental weights $\varpi_1,\ldots,\varpi_{n-1}$ and simple coroots $\al_1^{\vee}, \ldots, \al_{n-1}^{\vee}$ in accordance with the table in the end of Bourbaki \cite{Bou2}. We understand that $\varpi_n = 0$. Let $V = \C^n$ be the dual vector representation of $G$. According to Givental-Lee \cite{GL03}, we have
$$\ch \, V = [\cO_{\sB} ( - \varpi_1 )] + \sum_{i = 1}^{n-1} a^{\varpi_i} ( [\cO_{\sB} ( - \varpi_{i+1} )] ) \in qK_G ( \sB ),$$
where we have $a^{\varpi_i} = (1 - Q^{\al_i^{\vee}}) ( [\cO_{\sB} ( - \varpi_i) ] \star )^{-1} \in \mathrm{End} \, qK _G ( \sB )$. Let $L \subset G$ be a Levi subgroup. If we specialize $Q^{\al_{i}^{\vee}} = 0$ when $\al_i^{\vee} \not\in Q^{\vee}_{L,+}$, then the effect of $\ch \, V$ restricts to that of $qK_{L} ( \sB^{L} )$. When $\al_i^{\vee} \not\in Q^{\vee}_{L,+}$, the effect $a^{\varpi_i}$ becomes a character twist on $qK_{L} ( \sB^{L} )$.
\end{fex}

Here we warn that the definition of quantum $K$-groups, as well as the normalizations in Theorem \ref{flsurj} and Example \ref{exToda} are different from the main body of the paper for the sake of simplicity of expositions.

\medskip

The organization of this paper is as follows: After recalling preliminary stuffs in \S 1, we provide a certain collection of elements in the equivariant $K$-groups of semi-infinite flag manifolds (Theorem \ref{X-weight}) in \S 2. These collections are the ``reduced version" of line bundles, and the only non-trivial point is that we can divide the classes of line bundles properly. Using these elements, we provide (Proposition \ref{resGH}) a new system of generators of $K_G ( \Gr_G )_\lo$ in \S 3. In order to transplant elements from semi-infinite flag manifolds to affine Grassmannian, we prove Theorem \ref{fnc} (Theorem \ref{ncPet}). In \S 4, we prove Theorem \ref{flsurj} (Theorem \ref{qK-surj}), that is an essential tool to compute the ``leading terms" of the the maps in Theorem \ref{fmain}. Using them, we prove Theorem \ref{fmain} in \S 5. In Appendix A, we present an another proof of Theorem \ref{flsurj} (Theorem \ref{qK-surj-gen}) that applies in much greater generality.

\section{Preliminaries}

A vector space is always a $\C$-vector space, and a graded vector space refers to a $\Z$-graded vector space whose graded pieces are finite-dimensional and its grading is bounded from the above. Tensor products are taken over $\C$ unless stated otherwise. We define the graded dimension of a graded vector space as
$$\mathrm{gdim} \, M := \sum_{i\in \Z} q^i \dim _{\C} M_i \in \Q (\!(q^{-1})\!).$$
We set $\C_q := \C [q,q^{-1}]$. As a rule, we suppress $\emptyset$ and associated parenthesis from notation. This particularly applies to $\emptyset = \tJ \subset \tI$ frequently used to specify parabolic subgroups.

\subsection{Groups, root systems, and Weyl groups}\label{sec:setup}\label{setup}
Basically, material presented in this subsection can be found in \cite{CG97, Kum02}.

Let $G$ be a connected, reductive algebraic group over $\C$ such that $[G,G]$ is a simply connected group of rank $r$ and we have a complementary torus $H'$ such that $G \cong [G,G] \times H'$. Let $B$ and $H$ be a Borel subgroup and a maximal torus of $G$ such that $H \subset B$. We set $N$ $(= [B,B])$ to be the unipotent radical of $B$. We denote the Lie algebra of an algebraic group by the corresponding German small letter. We have a (finite) Weyl group $W := N_G ( H ) / H$. For an algebraic group $E$, we denote its set of $\C [z]$-valued points by $E [z]$, its set of $\C [\![z]\!]$-valued points by $E [\![z]\!]$, and its set of $\C (z)$-valued points by $E (z)$. Let $\mathbf I \subset G [\![z]\!]$ be the preimage of $B \subset G$ via the evaluation at $z = 0$ (the Iwahori subgroup of $G [\![z]\!]$).

Let $\bX^* := \mathrm{Hom} _{gr} ( H, \Gm )$ be the weight lattice of $H$, and let $\bX^* ( G )$ denote the subgroup of $\bX^*$ whose elements define characters of $G$. We set $\bX_*$ and $\bX_* ( G )$ as the dual lattices of $\bX^*$ and $\bX^* ( G )$, respectively. We denote the natural pairings of lattices by $\left< \bullet, \bullet \right>$.

Let $\Delta \subset \bX^*$ be the set of roots, let $\Delta_+ \subset \Delta$ be the set of roots that yield root subspaces in $\gb$, and let $\Pi \subset \Delta _+$ be the set of simple roots. We set $\Delta_- := - \Delta_+$. Let $Q^{\vee} \subset \bX_*$ be the $\Z$-span of coroots. We define $\Pi^{\vee} \subset Q ^{\vee}$ to be the set of positive simple coroots, and let $Q_+^{\vee} \subset Q ^{\vee}$ be the set of non-negative integer span of $\Pi^{\vee}$. For $\beta, \gamma \in \bX_*$, we define $\beta \ge \gamma$ if and only if $\beta - \gamma \in Q^{\vee}_+$. Let $\tI := \{1,2,\ldots,r\}$. We fix bijections $\mathtt I \cong \Pi \cong \Pi^{\vee}$ such that $i \in \tI$ corresponds to $\alpha_i \in \Pi$, its coroot $\alpha_i^{\vee} \in \Pi ^{\vee}$, and a simple reflection $s_i \in W$ corresponding to $\alpha_i$. We also have a reflection $s_{\alpha} \in W$ corresponding to $\alpha \in \Delta_+$. For each $\tJ \subset \tI$, we set $\bX^*_+ ( \tJ ) := \{ \lambda \in \bX^* \mid \left< \alpha^{\vee}_i, \lambda \right> \ge 0, \hskip 2mm \forall i \in \tJ \}$. Let $\{\varpi_i\}_{i \in \tI} \subset \bX^*_+$ be the set of fundamental weights (i.e. $\left< \al_i^{\vee}, \varpi_j \right> = \delta_{i,j}$) and we set $\rho := \sum_{i \in \tI} \varpi_i = \frac{1}{2}\sum_{\al \in \Delta^+} \al \in \bX^*_+$.

For a subset $\tJ \subset \tI$, we define $P ^\tJ$ to be the standard parabolic subgroup of $G$ corresponding to $\tJ$. I.e. we have $\gb \subset \gp ^\tJ \subset \g$ and $\gp ^\tJ$ contains the root subspace corresponding to $- \alpha_i$ ($i \in \tI$) if and only if $i \in \tJ$. Then, the set of characters of $P ^ \tJ$ is identified with $\bX^*_0 ( \tJ ) := \bX^* ( G ) \oplus \Lambda^{( \tI \setminus \tJ )}$, where we set $\Lambda^\tJ := \sum_{i \in \tJ} \Z \varpi_i $. We also set
$$\La^{\tJ}_{++} := \sum_{j \in \tJ} \Z_{> 0} \varpi_j \subset \La^{\tJ}_{+} := \sum_{j \in \tJ} \Z_{\ge 0} \varpi_j \subset \bX^*, \hskip 3mm  Q_{\tJ,+}^{\vee} := \sum_{j \in \tJ} \Z_{\ge 0} \al_j^{\vee} \subset Q_{\tJ}^{\vee} := \sum_{j \in \tJ} \Z \al_j^{\vee}.$$
We define $W ^\tJ \subset W$ to be the subgroup generated by $\{s_i\}_{i \in \tJ}$. It is the Weyl group of the maximal reductive subgroup $L ^\tJ$ of $P ^ \tJ$ that contains $H$ (we refer $L^\tJ$ as the standard Levi subgroup of $P^\tJ$ in the below).

Let $\la \in \bX^*$. We consider the subset
$$\Sigma ( \la ) := \text{convex span of } \{ W \la \} \subset \bX^* \otimes _\Z \R.$$
We set $\Sigma_* ( \la ) := \Sigma ( \la ) \setminus \{ W \la \}$.

We set $\bG := G \times \Gm$, $\bL ^ \tJ  := L ^ \tJ  \times \Gm$, and $\bH := H \times \Gm$ for the simplicity of notation.

Let $\Delta_{\af} := \Delta \times \Z \delta \cup \{m \delta\}_{m \neq 0}$ be the untwisted affine root system of $\Delta$ with its positive part $\Delta_+ \subset \Delta_{\af, +}$. We set $\alpha_0 := - \vartheta + \delta$, $\Pi_{\af} := \Pi \cup \{ \alpha_0 \}$, and $\tI_{\af} := \tI \cup \{ 0 \}$, where $\vartheta$ is the highest root of $\Delta_+$. We set $W _{\af} := W \ltimes Q^{\vee}$ and call it the affine Weyl group. It is a reflection group generated by $\{s_i \mid i \in \tI_{\af} \}$, where $s_0$ is the reflection with respect to $\alpha_0$. Let $\ell : W_\af \rightarrow \Z_{\ge 0}$ be the length function and let $w_0^{\tJ} \in W$ be the longest element in $W^{\tJ} \subset W_\af$. We set $\widetilde{W}_\af := W \ltimes \bX_*$ and call it the extended affine Weyl group. We have $t_{\beta} \in \bX_* \subset \widetilde{W}_\af$ for each $\beta \in \bX_*$ such that $t_{\beta} \in W_\af$ for $\beta \in Q^{\vee}$, $u t_{\beta} u^{-1} = t_{u \beta}$ for each $u \in W$, and $t_{- \vartheta^{\vee}} := s_{\vartheta} s_0$ (for the coroot $\vartheta^{\vee}$ of $\vartheta$). By setting $$\ell ( w t_{\gamma} ) = \ell ( t_{\gamma} w ) = \ell ( w )$$
for $w \in W_\af$ and $\gamma \in \bX_* ( G )$, we extend the length function to $\widetilde{W}_\af$ (that is possible by $\bX_* \cong \bX_* ( G ) \times Q^{\vee}$).

Let $\le$ be the Bruhat order of $W_\af$. In other words, $w \le v$ holds if and only if a subexpression of a reduced decomposition of $v$ yields a reduced decomposition of $w$ (see \cite{BB05}). We define the generic (semi-infinite) Bruhat order $\le_\si$ as:
\begin{equation}
w \le_\si v \Leftrightarrow w t_{\beta} \le v t_{\beta} \hskip 5mm \text{for every } \beta \in Q^{\vee} \text{ such that } \left< \beta, \al_i \right> \ll 0 \text{ for } i \in \tI. \label{si-ord}
\end{equation}
By \cite{Lus80}, this defines a preorder on $W_{\af}$. Here we remark that $w \le v$ if and only if $w \ge_\si v$ for $w,v \in W$.

\begin{thm}[Peterson \cite{Pet97} Lecture 13]\label{sib}
Let $w \in W_\af$ be such that $w \le_\si e$. We have $w = u t_{\beta}$ for some $u \in W$ and $\beta \in Q^{\vee}_+$. \hfill $\Box$
\end{thm}

For $w, v \in \widetilde{W}_\af$, we write $w \ge_\si v$ if and only if there exists $\gamma \in \bX_* ( G )$ such that $w t_{\gamma}, v t_{\gamma} \in W_\af$ and $w t_{\gamma} \ge_\si v t_{\gamma}$.

Let $\widetilde{W}_\af^-$ denote the set of minimal length representatives of $\widetilde{W}_\af / W$ in $\widetilde{W}_\af$. We set
$$\bX _*^- ( \tJ ) := \{\beta \in \bX _* \mid \left< \beta, \al_i \right> < 0, \forall i \in \tJ \}$$
and
$$\bX _*^{\le} ( \tJ ) := \{\beta \in \bX _* \mid \left< \beta, \al_i \right> \le 0, \forall i \in \tJ \}.$$
We have $\bX _*^- ( \tJ ) \subset \bX _*^- ( \tJ' )$ and $\bX _*^{\le} ( \tJ ) \subset \bX _*^{\le} ( \tJ' )$ when $\tJ' \subset \tJ$.

\begin{thm}[see e.g. Macdonald \cite{Mac03}]\label{af-len}
For $\beta \in \bX _*^-$, it holds:
\begin{enumerate}
\item We have $\ell (ut_{\beta}) = \ell (t_{\beta}) - \ell ( u )$ and $\ell (t_{\beta} u ) = \ell (t_{\beta}) + \ell ( u )$ for every $u \in W$;
\item For each $u \in W$ and $\beta' \in \bX _*^{\le}$, we have
$$\hskip -8mm \ell ( t_{u\beta})= \ell ( u t_{\beta} u^{-1}) = \ell ( t_{\beta} )\hskip 3mm \text{and} \hskip 3mm \ell ( t_{u ( \beta + \beta' )} ) = \ell ( t_{u \beta} ) + \ell ( t _{u \beta'}) = 2 \left< \beta + \beta', \rho \right>;$$
\item Each $w \in \widetilde{W}_\af^-$ is decomposed into $w = u t_{\gamma}$ for some $u \in W$ and $\gamma \in \bX_*^{\le}$ such that $\ell ( w ) = \ell ( t_{\gamma} ) - \ell ( u )$.
\end{enumerate}
\end{thm}

\begin{proof}
The first assertions follow from \cite[(2.4.1)]{Mac03}. The second assertions follow from 1) and \cite[(2.4.2)]{Mac03}. The third assertion is a consequence of \cite[(2.4.3)]{Mac03}.
\end{proof}

For each $\la \in \bX^*_+ ( \tJ )$, we denote a finite-dimensional simple $P^{\tJ}$-module with a non-zero $B$-eigenvector $\bv_{\la}$ of $H$-weight $\la$ by $V ^\tJ ( \la )$. Let $R ( G )$ be the (complexified) representation ring of $G$. We have an identification $R ( G ) = ( \C [H] )^W \subset \C \bX^*$ by taking characters. For a semi-simple $H$-module $V$, we set
$$\ch \, V := \sum_{\la \in \bX^*} e^\la \cdot \dim _{\C} \mathrm{Hom}_H ( \C_\la, V ).$$
If $V$ is a $\Z$-graded $H$-module in addition, then we set
$$\gch \, V := \sum_{\la \in \bX, n \in \Z} q^n e^\la \cdot \dim _{\C} \mathrm{Hom}_H ( \C_\la, V_n ).$$

For a $\bH$-equivariant coherent sheaf on a projective $\bH$-variety $\mathcal X$, let $\chi ( \mathcal X, \mathcal F ) \in \C [\bH]$ denote its equivariant Euler-Poincar\'e characteristic. We set $\bX^*_\af := \bX^* \oplus \Z \delta$ and understand that $e ^{\delta} = q \in \C \bX^*_\af = \C [\bH]$.

For $\tJ' \subset \tJ \subset \tI$, we identify $W^{\tJ} / W^{\tJ'}$ with its minimal coset representative in $W^{\tJ}$. We set $\sB^{\tJ}_{\tJ'} := P^\tJ / P ^{\tJ'} $ and call it the partial flag manifold of $L^\tJ$. It is equipped with the Bruhat decomposition
$$\sB^{\tJ}_{\tJ'} = \bigsqcup_{w \in W^{\tJ} / W^{\tJ'}} \mathbb O^{\tJ}_{\tJ'} ( w )$$
into $B$-orbits such that $\mathrm{codim}_{\sB^{\tJ}_{\tJ'}} \bO^{\tJ}_{\tJ'} ( w ) = \ell ( w )$ for each $w \in W^{\tJ} / W^{\tJ'}$. We set $\sB^\tJ_{\tJ'} ( w ) := \overline{\bO^\tJ_{\tJ'} ( w )} \subset \sB^{\tJ}$.

We have a notion of $H$-equivariant $K$-group $K_H ( \sB^{\tJ}_{\tJ'} )$ of $\sB^{\tJ}_{\tJ'}$ with coefficients in $\C$ (see e.g. \cite{KK90}). Explicitly, we have
\begin{equation}
K_H ( \sB^{\tJ}_{\tJ'} ) = \bigoplus_{w \in W^{\tJ} / W^{\tJ'}} \C [H] \, [\cO _{\sB^{\tJ}_{\tJ'} (w)}].\label{desc-K}
\end{equation}

For each $\la \in w_0^\tJ \bX^*_0 ( \tJ' )$, we have a line bundle $\cO _{\sB^{\tJ}_{\tJ'}} ( \la )$ such that
$$\ch \, H ^0 ( \sB^{\tJ}_{\tJ'}, \cO_{\sB^{\tJ}_{\tJ'}} ( \la ) ) = \ch \, V^{\tJ} ( \la ), \hskip 3mm \cO_{\sB^{\tJ}_{\tJ'}} ( \la ) \otimes_{\cO_{\sB^{\tJ}_{\tJ'}}} \cO _{\sB^{\tJ}_{\tJ'}} ( - \mu ) \cong \cO_{\sB^{\tJ}_{\tJ'}} ( \la - \mu )$$
holds for $\la, \mu \in w_0 ^\tJ \bX^*_0 ( \tJ' ) \cap \bX^*_+ ( \tJ )$.

\subsection{The nil-DAHA and its spherical version}

\begin{defn}\label{def-DAHA}
The nil-DAHA $\sH_q$ or $\sH_q ( G )$ of type $G$ is a $\C_q$-algebra generated by $\{ e^{\la} \}_{\la \in \bX^*} \cup \{ D_i \} _{i \in \tI_\af} \cup \{ T_{\gamma} \}_{\gamma \in \bX_* (G)}$ subject to the following relations:
\begin{enumerate}
\item $e^{\la + \mu} = e ^{\la} \cdot e ^{\mu}$ for $\la, \mu \in \bX^*$;
\item $D_i ^2 = D_i$ for each $i \in \tI_\af$;
\item For each distinct $i, j \in \tI_{\af}$, we set $m_{i,j} \in \Z_{> 0}$ as the minimum number such that $(s_is_j)^{m_{i,j}} = 1$. Then, we have
$$\overbrace{D_i D_j \cdots}^{m_{i,j}\text{-terms}} = \overbrace{D_j D_i \cdots}^{m_{i,j}\text{-terms}};$$
\item For each $\la \in \bX^*$ and $i \in \tI_\af$, we have
$$D_i e^{\la} - e^{s_i \la} D_i = \frac{e^{\la} - e^{s_i \la}}{1 - e ^{\al_i}}, \hskip 5mm \text{where} \hskip 5mm e^{\al_0} = q \cdot e^{- \vartheta^{\vee}};$$
\item $T_{\gamma} T_{\gamma'} = T_{\gamma'} T_{\gamma}$ for each $\gamma, \gamma' \in \bX_* ( G )$;
\item $T_{\gamma} D_i = D_i T_{\gamma}$ for each $i \in \tI_\af$ and $\gamma \in \bX_* ( G )$;
\item $T_{\gamma} e^{\la} = q^{\left< \gamma, \la \right>} e^{\la} T_{\gamma}$ for each $\la \in \bX^*$ and $\gamma \in \bX_* ( G )$.
\end{enumerate}
We also consider the $\C_q$-subalgebras $\sH_q^0, \sH_q ( \tJ ) \subset \sH_q$ generated by $\{ D_i \mid i \in \tI_\af \}$ and $\{ e^{\la}, D_i \mid \la \in \bX^*, i \in \tJ\}$ (for $\tJ \subset \tI_\af$), respectively.
\end{defn}

Let $\sS_q' := \C [\bH] \otimes \C W_\af$ be the smash product algebra, whose multiplication reads as:
$$( e^\la \otimes w ) (e^{\mu} \otimes v) = e^{\la + w \mu} \otimes w v \hskip 5mm \la,\mu \in \bX^*_\af, w,v \in W_\af.$$
We add $1 \otimes t_{\gamma} \in \C \otimes \C \widetilde{W}_\af$ ($\gamma \in \bX_* ( G )$) such that
$$( e^\la \otimes t_{\gamma} ) (e^{\mu} \otimes t_{\gamma'}) = q^{\left< \gamma, \mu \right>} e^{\la + \mu} \otimes t_{\gamma + \gamma'} \hskip 5mm \la,\mu \in \bX^*_\af, \gamma, \gamma' \in \bX_* ( G )$$
to $\sS_q'$ to obtain the smash product algebra $\sS_q := \C [\bH] \otimes \C \widetilde{W}_\af$. Let $\C ( \bH )$ denote the fraction field of (the Laurant polynomial algebra) $\C [\bH]$. We have a scalar extension
$$\sR_q := \C ( \bH ) \otimes_{\C [\bH]} \sS_q = \C ( \bH ) \otimes_{\C} \C \widetilde{W}_\af.$$

The following is a very slight extension of \cite{LSS10} \S 2.2 (and hence we omit its proof):

\begin{thm}[cf. \cite{LSS10} \S 2.2]
We have an embedding of algebras $\imath^* : \sH_q \hookrightarrow \sR_q$:
\begin{align*}
e^{\la} \mapsto e ^{\la} \otimes 1, \,\, & D_i \mapsto  \frac{1}{1 - e ^{\al_i}} \otimes 1 - \frac{e^{\al_i}}{1 - e ^{\al_i}} \otimes s_i, \, T_{\gamma} \mapsto 1 \otimes t_{\gamma}. 
\end{align*}
for each $\la \in \bX^*_\af, i \in \tI_\af,$ and $\gamma \in \bX_* ( G )$.
\end{thm}

\begin{cor}[Leibniz fule for $D_i$]\label{Lei}
Let $i \in \tI_\af$ and $\la \in \bX^*_\af$. We have
$$D_i \cdot e^{\la} = \frac{e^{\la} - e^{s_i\la}}{1 - e^{\al_i}} + e^{s_i \la} \cdot D_i \hskip 3mm \text{in} \hskip 3mm \sR_q.$$
\end{cor}

Since we have a natural action of $\sR_q$ on $\C ( \bH )$, we obtain an action of $\sH_q$ on $\C ( \bH )$ (in a way it preserves $\C [\bH]$), that we call the polynomial representation.

For $w \in t_{\gamma} W_\af$ ($\gamma \in \bX_* ( G )$), we find a reduced expression $w = t_{\gamma} s_{i_1} \cdots s_{i_{\ell}}$ ($i_1,\ldots,i_{\ell} \in \tI_\af$) and set
$$D_w := T_{\gamma} D_{s_{i_1}} D_{s_{i_2}} \cdots D_{s_{i_{\ell}}} \in \sH_q.$$
By Definition \ref{def-DAHA} 3), the element $D_w$ is independent of the choice of a reduced expression. By Definition \ref{def-DAHA} 2), we have
$D_i D_{w_0} = D_{w_0}$ for each $i \in \tI$, and hence $D_{w_0}^2 = D_{w_0}$. We have an explicit form
\begin{equation}
D_{w_0} = 1 \otimes \left( \sum_{w \in W} w \right) \cdot \frac{e^{-\rho}}{\prod_{\al \in \Delta^+} ( e^{- \al / 2} - e^{\al/2})} \otimes 1 \in \sA_q \label{WCF}
\end{equation}
obtained from the (left $W$-invariance of the) Weyl character formula. We set
$$\sH_q^\sh \equiv \sH_q^\sh ( G ) := D_{w_0} \sH_q D_{w_0}$$
and call it the spherical nil-DAHA of type $G$.

\begin{thm}[see e.g. Kostant-Kumar \cite{KK90}]\label{finK}
We have a $\sH_q ( \tI )$-action on $K_{\bH} ( \sB )$ with the following properties:
\begin{enumerate}
\item For each $\la \in \bX^*$, the left multiplication by $e^{\la} \in \sH_q ( \tI )$ is equal to the $H$-character twist of $K_{\bH} ( \sB )$ by $e^{\la}$;
\item For each $i \in \tI$, we have
$$D_i ( [\cO_{\sB ( w )}] ) = \begin{cases}[\cO_{\sB ( s_i w )}] & (s_i w < w)\\ [\cO_{\sB ( w )}] & (s_i w > w)\end{cases};$$
\item For $\la \in \bX^*$, the twist by $\cO_\sB ( \la )$ defines a $\sH_q ( \tI )$-module automorphism;
\item We have $K_{\bG} ( \sB ) = D_{w_0} K_{\bH} ( \sB )$;
\item We have $K_{\bH} ( \sB ) = \sH_q ( \tI ) \cdot [\cO_{\sB}] = \C_q [H] \cdot K_{\bG} ( \sB ) \subset K_{\bH} ( \sB )$.
\end{enumerate}
\end{thm}

\begin{cor}\label{K-rest}
For each $\tJ' \subset \tJ \subset \tI$, we have a $\sH_q ( \tJ' )$-module map
$$K_{\bH} ( \sB^{\tJ} ) \longrightarrow K_{\bH} ( \sB^{\tJ'} )$$
that sends $[\cO_{\sB^{\tJ}} ( \la )]$ to $[\cO_{\sB^{\tJ'}} ( \la )]$ for every $\la \in \bX^*$.
\end{cor}

\begin{proof}
We have an algebra map $K_{\bL^{\tJ}} ( \sB^{\tJ} ) \longrightarrow K_{\bL^{\tJ'}} ( \sB^{\tJ'} )$ 
that sends $[\cO_{\sB^{\tJ}} ( \la )]$ to $[\cO_{\sB^{\tJ'}} ( \la )]$ for every $\la \in \bX^*$. It is invariant under the action of $D_j$ for $j \in \tJ'$ by Theorem \ref{finK} 3). By extending the scalar, we obtain a map $K_{\bH} ( \sB^{\tJ} ) \longrightarrow K_{\bH} ( \sB^{\tJ'} )$. By the Leibniz rule, this map commutes with the $D_i$-actions for each $i \in \tJ'$. Thus, it gives rise to a $\sH_q ( \tJ' )$-module map as required.
\end{proof}

\begin{cor}[\cite{KK90}]\label{Rsph}
For each $\tJ' \subset \tJ \subset \tI$, the pullback defines a subspace
$$K_{\bH} ( \sB^{\tJ}_{\tJ'} ) \cong K_{\bH} ( \sB^{\tJ} ) D_{w_0^{\tJ'}}  \subset K_{\bH} ( \sB^{\tJ} ).$$
\end{cor}

\subsection{Quasi-map spaces}\label{sec:QM}

Here we recall basics of quasi-map spaces from \cite{FM99,FFKM}.

We have $W$-equivariant isomorphism $H_2 ( \sB, \Z ) \cong Q ^{\vee}$. This identifies the (integral points of the) effective cone of $\sB$ with $Q_{+}^{\vee}$. A quasi-map $( f, D )$ is a map $f : \P ^1 \rightarrow \sB$ together with an $\tI$-colored effective divisor
$$D = \sum_{i \in \tI, x \in \P^1 (\C)} m_x (\alpha^{\vee}_i) \alpha^{\vee}_i \otimes [x] \in Q^{\vee} \otimes_\Z \mathrm{Div} \, \P^1 \hskip 3mm \text{with} \hskip 3mm m_x (\alpha^{\vee}) \in \Z_{\ge 0}.$$
We call $D$ the defect of $(f, D)$. We define the total defect of $(f, D)$ by
$$|D| := \sum_{i \in \tI, x \in \P^1 (\C)} m_x (\alpha^{\vee}_i) \alpha^{\vee}_i \in Q_+^{\vee}.$$

For each $\beta \in Q_+^{\vee}$, we set
$$\sQ ( \sB, \beta ) : = \{ f : \P ^1 \rightarrow X \mid \text{ quasi-map s.t. } f _* [ \P^1 ] + | D | = \beta \},$$
where $f_* [\P^1]$ is the class of the image of $\P^1$ multiplied by the degree of $\P^1 \to \mathrm{Im} \, f$. We denote $\sQ ( \sB, \beta )$ by $\sQ_G ( \beta )$ or $\sQ ( \beta )$ for simplicity.

\begin{defn}[Drinfeld-Pl\"ucker data]\label{Zas}
Consider a collection $\mathcal L = \{( \psi_{\lambda}, \mathcal L^{\lambda} ) \}_{\lambda \in \La_+}$ of inclusions $\psi_{\lambda} : \mathcal L ^{\lambda} \hookrightarrow V ( \lambda ) \otimes _{\C} \mathcal O _{\P^1}$ of line bundles $\mathcal L ^{\lambda}$ over $\P^1$. The data $\mathcal L$ is called a Drinfeld-Pl\"ucker data (DP-data) if the canonical inclusion of $G$-modules
$$\eta_{\lambda, \mu} : V ( \lambda + \mu ) \hookrightarrow V ( \lambda ) \otimes V ( \mu )$$
induces an isomorphism
$$\eta_{\lambda, \mu} \otimes \mathrm{id} : \psi_{\lambda + \mu} ( \mathcal L ^{\lambda + \mu} ) \stackrel{\cong}{\longrightarrow} \psi _{\lambda} ( \mathcal L^{\lambda} ) \otimes_{\cO_{\P^1}} \psi_{\mu} ( \mathcal L^{\mu} )$$
for every $\lambda, \mu \in \La_+$.
\end{defn}

\begin{thm}[Drinfeld, see Finkelberg-Mirkovi\'c \cite{FM99}]\label{Dr}
The variety $\sQ ( \beta )$ is isomorphic to the variety formed by isomorphism classes of the DP-data $\mathcal L = \{( \psi_{\la}, \mathcal L^{\la} ) \}_{\la \in \La_+}$ such that $\deg \, \mathcal L ^{\lambda} = \left< w_0 \beta, \la \right>$. In addition, $\sQ ( \beta )$ is an irreducible variety of dimension $\dim \, \sB + 2 \left< \beta, \rho \right>$.
\end{thm}

\begin{thm}[Braverman-Finkelberg \cite{BF14b}]\label{rat}
The variety $\sQ ( \beta )$ is a normal variety with rational singularities.
\end{thm}

For each $\la \in \bX^*$, and $\beta \in Q_+^{\vee}$, we have a $G$-equivariant line bundle $\cO _{\sQ ( \beta )} ( \lambda )$ obtained by the tensor product of the pull-backs $\cO _{\sQ ( \beta )}( \varpi_i )$ of the $i$-th $\cO ( 1 )$ via the embedding
\begin{equation}
\sQ ( \beta ) \hookrightarrow \prod_{i \in \mathtt I} \P ( V ( \varpi_i ) \otimes_{\C} \C [z] _{\le - \left< w_0 \beta, \varpi_i \right>} )\label{Pemb}
\end{equation}
and a $G$-character. We have $\chi ( \sQ ( \beta ), \cO_{\sQ} ( \la ) ) \in \C [\bH]$ for $\beta \in Q^{\vee}, \la \in \bX^*$,
where the grading $q$ is understood to count the degree of $z$ detected by the $\Gm$-action. Here we understand that $\chi ( \sQ ( \beta ), \cO_{\sQ ( \beta )} ( \la ) ) = 0$ if $\beta \not\in Q^{\vee}_+$.

We have an embedding $\sB \subset \sQ ( \beta )$ such that the line bundles $\cO ( \la )$ ($\la \in \bX^*$) correspond to each other by restrictions (\cite{BF14b,Kat18}).

\subsection{Graph and map spaces and their line bundles}\label{GQL}

We refer \cite{KM94,FP95,GL03} for the precise definitions of the notions appearing in this subsection.

For each non-negative integer $n$ and $\beta \in Q^{\vee}_+$, we set $\sGB_{n, \beta}$ to be the space of stable maps of genus zero curves with $n$-marked points to $( \P^1 \times \sB )$ of bidegree $( 1, \beta )$, that is also called the graph space of $\sB$. A point of $\sGB_{n, \beta}$ is a genus zero curve $C$ with $n$-marked points $\{x_1,\ldots,x_n\}$, together with a map to $\P^1$ of degree one. Hence, we have a unique $\P^1$-component of $C$ that maps isomorphically onto $\P^1$. We call this component the main component of $C$ and denote it by $C_0$. For a genus zero curve $C$, let $|C|$ denote the number of its irreducible components. The space $\sGB_{n, \beta}$ is a normal projective variety by \cite[Theorem 2]{FP95} that have at worst quotient singularities arising from the automorphism of curves (in particular, they have rational singularities). The natural $\bH$-action on $( \P^1 \times \sB )$ induces a natural $\bH$-action on $\sGB_{n, \beta}$. Moreover, $\sGB_{0, \beta}$ has only finitely many isolated $\bH$-fixed points, and thus we can apply the formalism of Atiyah-Bott-Lefschetz localization (cf. \cite[p200L26]{GL03} and \cite[Proof of Lemma 5]{BF14b}).

We have a morphism $\pi_{n, \beta} : \sGB_{n, \beta} \rightarrow \sQ ( \beta )$ that factors through $\sGB_{0, \beta}$ (Givental's main lemma \cite{Giv96}; see \cite[\S 8]{FFKM} and \cite[\S 1.3]{FP95}). Let $\widetilde{\mathtt{ev}}_j : \sGB_{n, \beta} \to \P^1 \times \sB$ ($1 \le j \le n$) be the evaluation at the $j$-th marked point, and let $\mathtt{ev}_j : \sGB_{n, \beta} \to \mathscr B$ be its composition with the second projection. The variety $\sGB_{n,\beta}$ is irreducible (as a special feature of flag varieties, see \cite[\S 1.2]{FP95} and \cite{KP01}).

Let $\sX ( \beta ) \subset \sGB_{2,\beta}$ denote the subscheme such that the first marked point projects to $0 \in \P^1$, and the second marked point projects to $\infty \in \P^1$ through the first projection of $\P^1 \times \sB$. By abuse of notation, we write the restriction of $\mathtt{ev}_i$ ($i =1,2$) to $\sX ( \beta )$ by the same letter. Let $\pi_{\beta} : \sX ( \beta ) \rightarrow \sQ ( \beta )$ be the restriction of $\pi_{2,\beta}$ to $\sX ( \beta )$. 
In view of Theorem \ref{rat}, the morphism $\pi_{\beta}$ is the rational resolution of singularities in an orbifold sense.

For each $\la \in \bX^*$, we have a line bundle $\cO_{\sX (\beta)} ( \la ) := \pi_{\beta}^* \cO _{\sQ ( \beta )}( \la )$. In case we want to stress the group $G$, we write $\sX _G ( \beta )$ instead of $\sX (\beta)$.

\subsection{Equivariant quantum $K$-group of $\sB$}\label{eqK}

We introduce a polynomial ring $\C Q^{\vee}_+$ and the formal power series ring $\C [\![Q^{\vee}_+]\!]$ with their variables $Q_i = Q^{\al_i^{\vee}}$ ($i \in \tI$). We set $Q^{\beta} := \prod_{i \in \tI} Q_i ^{\left< \beta, \varpi_i \right>}$ for each $\beta \in Q^{\vee}$. We define the $\bG$-equivariant (small) quantum $D_q$-module of $\sB$ as:
\begin{equation}
qK_{\bG} ( \sB ) := K_{\bG} ( \sB ) \otimes \C Q^{\vee}_+.\label{qDdef}
\end{equation}
Note that the specialization $q = 1$ yields
\begin{equation}
qK_{G} ( \sB ) := K_{G} ( \sB ) \otimes \C Q^{\vee}_+.\label{qKdef}
\end{equation}
Let $qK_{\bG} ( \sB )^{\wedge}$ and $qK_{G} ( \sB )^{\wedge}$ denote the completions of $qK_{\bG} ( \sB )$ and $qK_{G} ( \sB )$ with respect to the variables $\{ Q_i \}_{i \in \tI}$.

Let $\left< \bullet, \bullet \right>^{\GW}$ be the $R ( \bG )$-linear pairing on $qK_{\bG} ( \sB )^{\wedge}$ defined as:
$$\left< a, b \right>^{\GW} := \sum_{\beta \in Q^{\vee}_+} \chi ( \sX ( \beta ), \mathrm{ev}_1^* a \otimes \mathrm{ev}_1^* b ) Q^{\beta} \in \C [\bH] [\![Q^{\vee}_+]\!] \hskip 5mm a, b \in qK_{\bG} ( \sB )^{\wedge}.$$
Since the specialization $Q^{\beta} = 0$ ($\beta \neq 0$) recovers the ($\bG$-equivariant) Euler-Poincar\'e pairing of $\sB$, we know that $\left< \bullet, \bullet \right>^{\GW}$ is non-degenerate. For each $\la \in \bX^*$, the bilinear functional
$$\left< a, b \right>^{\GW}_{\la} := \sum_{\beta \in Q^{\vee}_+} \chi ( \sX ( \beta ), \pi_{\beta}^* \cO_{\sQ ( \beta )} ( \la ) \otimes \mathrm{ev}_1^* a \otimes \mathrm{ev}_1^* b ) Q^{\beta} \in \C [\bH] [\![Q^{\vee}_+]\!]$$
induces a(n unique) linear operator $A^{\la} ( \bullet )$ on $qK_{\bG} ( \sB )^{\wedge}$ such that
$$\left< A^{\la} a, b \right>^{\GW} = \left< a, b \right>^{\GW}_\la \hskip 5mm a, b \in qK_{\bG} ( \sB )^{\wedge}.$$
We remark that the operator $A^\la$ is the character twist when $\la \in \bX^* ( G )$. In case we want to stress the dependence on $G$, we write $\left< \bullet, \bullet \right>^{\GW}_G$ and $A^{\la}_G$ instead of $\left< \bullet, \bullet \right>^{\GW}$ and $A^{\la}$, respectively.

\begin{thm}[Iritani-Milanov-Tonita \cite{IMT15} and \cite{Kat18c}]\label{reconst}
We have:
\begin{enumerate}
\item For $\la, \mu \in \bX^*$, we have $A^\la \circ A^{\mu} = A^{\la + \mu}$ in $\End_{R ( \bG )} ( qK_{\bG} ( \sB )^{\wedge} )$;
\item For $\la \in \bX^*$ and $c \in K_{\bG} ( \sB ) \otimes 1 \subset qK_{\bG} ( \sB )$, we have
$$A^{\la} c \equiv \cO_{\sB} ( \la ) \otimes_{\cO_\sB} c \mod ( Q_i \mid i \in \tI);$$
\item The $q = 1$ specialization of the operator $A^{- \varpi_i}$ $(i \in \tI)$ is the quantum multiplication by $[\cO_{\sB} ( - \varpi_i )]$ on $qK_{G} ( \sB )$;
\item The $R ( G )$-action, the $\C Q^{\vee}$-action, together with the quantum multiplications by $[\cO_{\sB} ( - \varpi_i )]$ $(i \in \tI)$, generates $qK_{G} ( \sB )$ as a ring;
\item For $f \in \C_q [A^{\la},Q^{\beta}\mid \la \in \bX^*, \beta \in Q^{\vee}_+ ]$, we have $f [\cO_{\sB}] = 0$ in $qK_{\bG} ( \sB )$ if and only if
$$\left< f [\cO_{\sB}], [\cO_{\sB}] \right>^\GW_\la = 0 \hskip 5mm \la \in \La_+.$$
\end{enumerate}
\end{thm}

\begin{proof}
The first two assertions follows from \cite{IMT15} Proposition 2.13 and Proposition 2.10, respectively. The third assertion is \cite[Lemma 6]{ACT18} (or \cite[Theorem 4.2]{Kat18c}). The fourth assertion is a consequence of the finiteness of quantum $K$-groups, seen in \cite[Proposition 9]{ACT18} and \cite[Corollary 3.3]{Kat18c}. The fifth assertion can be read off from the proof of \cite[Theorem 3.11]{Kat18c}.
\end{proof}

\section{Preparatory results}

\subsection{Affine Grassmanians}
We define our (thin) affine Grassmannian and (thin) flag manifold by
$$\Gr_G := G (\!(z)\!) / G [\![z]\!] \hskip 3mm \text{and} \hskip 3mm X_G := G (\!(z)\!) / \bI,$$
respectively. We have a natural map $\pi : X_G \rightarrow \Gr_G$ whose fiber is isomorphic to $\sB$. By \cite[\S 4.6]{BD} (cf. \cite[\S 2]{MV07}), the sets of connected components of $\Gr_G$ and $X_G$ are in bijection with $\bX_* ( G )$. Here we note that our assumption on $G$ guarantees that all connected components of $\Gr_G$ are mutually isomorphic as ind-varieties with $G [\![z]\!]$-actions.

\begin{thm}[Bruhat decomposition, \cite{Kum02} Corollary 6.1.20] We have $\bI$-orbit decompositions
$$\Gr_G = \bigsqcup_{\beta \in \bX_*} \mathring{\Gr_G} ( \beta ) \hskip 3mm \text{and} \hskip 3mm X = \bigsqcup_{w \in \widetilde{W}_\af} \bO_G^\af ( w )$$
with the following properties:
\begin{enumerate}
\item we have $\bO_G^\af ( v ) \subset \overline{\bO_G^\af (w)}$ if and only if $v \le w$;
\item $\pi ( \bO _G^\af ( w ) ) \subset \mathring{\Gr_G} ( \beta )$ if and only if $w \in t_{\beta} W$.\hfill $\Box$
\end{enumerate}
\end{thm}

Let us set $\Gr_G ( \beta ) := \overline{\mathring{\Gr_G} ( \beta )}$ and $X_w := \overline{\bO_G^\af ( w )}$ for $\beta \in \bX_*$ and $w \in \widetilde{W}_\af$. For $w \in \widetilde{W}_\af^-$, we also set $\Gr_G ( w ) := \Gr_G ( \beta )$ for $\beta \in \bX_*$ such that $w \in t_{\beta} W$.

We set
$$K_{\bH} ( \Gr_G ) := \bigoplus_{\beta \in \bX_*} \C [\bH] \, [\cO_{\Gr_G (\beta)}] \hskip 3mm \text{and} \hskip 3mm K_{\bH} ( X_G ) := \bigoplus_{w \in \widetilde{W}_\af} \C [\bH] \, [\cO_{X_w}].$$

The following is an affine version of Theorem \ref{finK}:

\begin{thm}[Kostant-Kumar \cite{KK90}]\label{sH-reg}
The vector space $K_{\bH} ( X_G )$ affords a regular representation of $\sH_q$ such that:
\begin{enumerate}
\item the subalgebra $\C [\bH] \subset \sH_q$ acts by the multiplication of the coefficients;
\item we have $D_i [\cO_{X_w}] = [\cO_{X_{s_i w}}]$ $(s_i w > w)$ or $[\cO_{X_{w}}]$ $(s_i w < w)$.\hfill $\Box$
\end{enumerate}
\end{thm}

Being a regular representation, we sometimes identify $K_\bH ( X_G )$ with $\sH_q$ (through $e^{\la} [\cO_{X_w}] \leftrightarrow e^{\la} D_w$ for $\la \in \bX_*^\af, w \in \widetilde{W}_\af$) and consider product of two elements in $\sH_q \cup K _\bH ( X_G )$. We may denote this product on $K_\bH ( X_G )$ by $\odot_q$.

\begin{thm}[Kostant-Kumar \cite{KK90}]\label{sH-sph}
The pullback defines an inclusion map $\pi^* : K_\bH ( \Gr_G ) \hookrightarrow K _\bH ( X_G )$ such that
$$\pi^* [\cO_{\Gr_G (\beta)}] = [X_{t_{\beta}}] D_{w_0} \hskip 5mm \beta \in Q^{\vee}.$$
In particular, $\mathrm{Im} \, \pi^* = \sH_q \odot_q D_{w_0}$ is a $\sH_q$-submodule.\hfill $\Box$
\end{thm}

\begin{thm}\label{LSS-formula}
Let $w \in \widetilde{W}_\af^-$ and let $\beta \in \bX_*^-$. We have
$$\pi^* [\cO_{\Gr_G (w)}] \odot_q \pi^* [\cO_{\Gr_G (\beta)}] = \pi^* [\cO_{\Gr_G (w t_\beta)}].$$
\end{thm}

\begin{proof}
We have $\ell ( t_{\beta} ) = \ell ( w_0 ) + \ell ( w_0 t_{\beta} )$ by Theorem \ref{af-len} 1). We have $w = u t_{\gamma}$ for some $u \in W$ and $\gamma \in \bX_*^{\le}$ such that $\ell ( w ) = \ell ( t_{\gamma} ) - \ell ( u )$ by Theorem \ref{af-len} 3). Now we have $\ell ( u t_{\gamma + \beta} ) = \ell ( w ) + \ell ( t_{\beta} )$ by Theorem \ref{af-len} 2). From these, the assertion follows by Theorem \ref{sH-reg} and Theorem \ref{sH-sph}.
\end{proof}

Theorem \ref{LSS-formula} implies that the set
$$\{ \pi^* [\cO_{\Gr_G ( \beta )}] \mid \beta \in \bX_*^- \} \subset ( K_{\bH} ( \Gr_G ), \odot_q )$$
forms a multiplicative system with respect to the right action. We denote by $K_{\bH} ( \Gr_G )_{\lo}$ the localization of $K_{\bH} ( \Gr_G )$ with respect to this right action. The action of an element $[\cO_{\Gr_G (\beta)}]$ on $K_{\bH} ( \Gr_G )$ in Theorem \ref{LSS-formula} is torsion-free, and hence we have an embedding $K_{\bH} ( \Gr_G ) \subset K_{\bH} ( \Gr_G )_\lo$. Since the left $\sH_q$-module structure on $( K_{\bH} ( \Gr_G ), \odot_q )$ commutes with this right action, we conclude that $K_{\bH} ( \Gr_G )_{\lo}$ is a $\sH_q$-module that contains $K_{\bH} ( \Gr_G )$.

\begin{cor}\label{h-op}
Let $i \in \tI$. For $\beta \in \bX_*^-$, we set
$$\bh _i := \pi^* [\cO_{\Gr_G (s_i t_{\beta} )}] \odot_q \pi^* [\cO_{\Gr_G ( t_{\beta} )}]^{-1}.$$
Then, the element $\bh_i$ is independent of the choice of $\beta$.
\end{cor}

\begin{proof}
By Theorem \ref{LSS-formula}, we have
\begin{align*}
[\cO_{\Gr_G (s_i t_{\gamma + \beta} )}] \odot_q [\cO_{\Gr_G (t_{\gamma + \beta} )}]^{-1} & = [\cO_{\Gr_G ( s_i t_{\beta} )}] \odot_q [\cO_{\Gr_G (t_{\gamma})}] \odot_q  [\cO_{\Gr_G (t_{\gamma})}]^{-1} \odot_q [\cO_{\Gr_G ( t_\beta )}]^{-1}\\
& = [\cO_{\Gr_G (s_i t_{\beta})}] \odot_q [\cO_{\Gr_G (t_{\beta})}]^{-1}
\end{align*}
for $\gamma \in \bX_*^-$. Hence, we conclude the assertion.
\end{proof}

In the below, we may drop $\pi^*$ in the notation and consider
$$K_\bG ( \Gr_G ) = D_{w_0} K_\bH ( \Gr_G ) \cong D_{w_0} K _\bH ( X_G ) D_{w_0} \subset K _\bH ( X_G )$$
as a subalgebra of $K _\bH ( X_G )$. Note that $[\cO_{\Gr_G ( \beta )}] \in K_\bG ( \Gr_G )$ for $\beta \in \bX_*^{-}$. In addition, $[\cO_{\Gr_G ( 0 )}]$ is the multiplicative unit of $K_{\bG} ( \Gr_G )$, and we sometimes denote it by $1$. It is clear that $K_\bG ( \Gr_G )$ affords a regular representation of $\sH_q^{\sh}$.

For each $\gamma \in \bX_*$, we can write $\gamma = \beta_1 - \beta_2$, where $\beta_1,\beta_2 \in \bX_*^-$. In particular, we have an element
$$\mt _{\gamma} := [\cO_{\Gr_G (t_{\beta_1})}] \odot_q [\cO_{\Gr_G (t_{\beta_2} )}]^{-1}.$$

\begin{lem}
For each $\gamma \in Q^{\vee}$, the element $\mt _{\gamma} \in K_{\bG} ( \Gr_G )_{\lo}$ is independent of the choices involved.
\end{lem}

\begin{proof}
Similar to the proof of Corollary \ref{h-op}. The detail is left to the readers. 
\end{proof}

\subsection{Semi-infinite flag manifolds}\label{sif}

In this subsection, we assume that $G$ is a simple algebraic group. This assumption implies $\Lambda = \bX^*$, $Q^{\vee} = \bX_*$, and $W_\af = \widetilde{W}_\af$. In \cite{Kat18d}, we have exhibited an ind-scheme $\bQ_G^{\ra}$ of ind-infinite type that is universal among these whose set $\C$-valued points are $G (\!(z)\!) / ( H \cdot N (\!(z)\!) )$. It is equipped with a $G (\!(z)\!)$-equivariant line bundle $\cO _{\bQ_G^{\ra}} ( \la )$ for each $\la \in \bX^*$. Here we normalized the label of line bundles such that $\Gamma ( \bQ_G^{\ra}, \cO_{\bQ_G^{\mathrm{rat}}} ( \la ) )$ is co-generated by its $H$-weight $\la$-part as a $B (\!(z)\!)$-module.

\begin{thm}[\cite{FM99,FFKM}]\label{si-Bruhat}
We have an $\bI$-orbit decomposition
$$\bQ_G^{\mathrm{rat}} = \bigsqcup_{w \in W_\af} \bO ( w )$$
with the following properties:
\begin{enumerate}
\item each $\bO ( w )$ has infinite dimension and infinite codimension in $\bQ_G^{\mathrm{rat}}$;
\item the right action of $\gamma \in Q^{\vee}$ on $\bQ_G^{\mathrm{rat}}$ yields the translation $\mathbb O ( w ) \mapsto \bO ( w t_{\gamma})$;
\item we have $\bO ( w ) \subset \overline{\bO ( v )}$ if and only if $w \le_{\si} v$.\hfill $\Box$
\end{enumerate}
\end{thm}

We define a $\C [\bH]$-module $K_{\bH} ( \bQ_G^{\ra} )$ as:
$$K_{\bH} ( \bQ_G^{\ra} ) := \{ \sum_{w \in W_{\af}} a_{w} [\cO_{\bQ_G ( w )}] \mid a_w \in \C [\bH], \, \exists \beta_0 \in Q^{\vee} \text{ s.t. } a _{u t_{\beta}} = 0, \, \forall u \in W, \beta \not> \beta_0 \},$$
where the sum in the definition is understood to be formal (i.e. we allow infinite sums). We define its subset
$$K_{\bH} ( \bQ_G ( t_{\beta} ) ) := \{ \sum_{w \in W_{\af}} a_{w} [\cO_{\bQ_G ( w )}] \mid a_w \in \C [\bH] \, \text{ s.t. } a _{u t_{\gamma}} = 0, \, \forall u \in W, \gamma \not\ge \beta \}$$
for each $\beta \in Q^{\vee}$. Employing the family $\{ K_{\bH} ( \bQ_G ( t_{\beta} ) ) \}_{\beta \in Q^{\vee}}$ of subsets of $K_{\bH} ( \bQ_G^{\mathrm{rat}} )$ as an open base of $0$, we obtain a topology on $K_{\bH} ( \bQ_G^{\ra} )$.

\begin{thm}[\cite{KNS17} Theorem 6.5]\label{H-si}
The vector space $K_{\bH} ( \bQ_G^{\mathrm{rat}} )$ affords a representation of $\sH_q$ such that:
\begin{enumerate}
\item the subalgebra $\C [\bH] \subset \sH_q$ acts by the multiplication as $\C [\bH]$-modules;
\item we have
$$D_i ( [\cO_{\bQ_G (w)}] ) = \begin{cases} [\cO_{\bQ_G (s_i w)}] & (s_i w >_{\si} w) \\ [\cO_{\bQ_G (w)}] & (s_i w <_{\si} w)\end{cases}.$$
\end{enumerate}
\end{thm}

For each $\beta \in Q^{\vee}$, we set
$$K_{\bG} ( \bQ_G^{\mathrm{rat}} ) := D_{w_0} ( K_{\bH} ( \bQ_G^{\mathrm{rat}} ) ) \hskip 3mm \text{and} \hskip 3mm K_{\bG} ( \bQ_G ( t_{\beta}) ) := D_{w_0} ( K_{\bH} ( \bQ_G ( t_{\beta}) ) ).$$

From the description of Theorem \ref{H-si}, we deduce that the right $Q^{\vee}$-action gives $\sH_q$-module endomorphisms of $K_\bH ( \bQ_G^{\mathrm{rat}} )$. We denote this endomorphism for $\beta \in Q^{\vee}$ by $Q^{\beta}$. It gives rise to an endomorphism of $K_{\bG} ( \bQ_G^{\mathrm{rat}} )$. 
We set $\C_q (\!(Q^{\vee})\!) := \C_q Q^{\vee} \otimes _{\C_q Q^{\vee}_+} \C_q [\![Q^{\vee}_+]\!]$. The commutative rings $\C_q Q^{\vee}$ and $\C_q (\!(Q^{\vee})\!)$ act on $K_{\bG} ( \bQ_G^{\ra} )$ from the right.

\begin{thm}\label{HQc}
For each $\la \in \La$, the $\C [\bH]$-linear extension of the assignment
$$[\cO _{\bQ_G ( w )}] \mapsto [\cO _{\bQ_G ( w )} ( \la )] \in K _\bH ( \bQ_G^{\mathrm{rat}} ) \hskip 3mm w \in W_\af$$
defines a $\sH_q$-module automorphism $($that we call $\Xi ( \la ))$. In addition, we have:
\begin{enumerate}
\item $\Xi ( \la )\circ \Xi ( \mu ) = \Xi ( \la + \mu )$ for $\la, \mu \in \La$;
\item $[\cO _{\bQ_G ( w )} ( \la )] = e ^{w \la} [\cO _{\bQ_G ( w )}] + \sum_{v <_\si w} a_w^v ( \la ) [\cO _{\bQ_G ( v )}]$ for $a_w^v \in \C [\bH]$;
\item The coefficients $a_w^v$ belongs to a $\C_q$-span of $\{e^\mu \}_{\mu \in \Sigma ( \la )}$;
\item $[\cO _{\sB ( w )} ( \la )] = e ^{w \la} [\cO _{\sB ( w )}] + \sum_{w < v \in W} a_w^v ( \la ) [\cO _{\sB ( v )}]$ for each $w \in W$.
\end{enumerate}
\end{thm}

\begin{proof}
The existence of the $\sH_q$-module structure and the assertion in the first item follow from \cite[Theorem 6.4]{KNS17} (though the definition of the $K$-groups are slightly different). The second item follows by \cite[Theorem 5.10]{KNS17} since a path with the equal initial/final directions is unique, and the path interpretation of coefficients $a_w^v$ automatically imposes order relation $v <_\si w$ (see \cite[\S 2.3]{KNS17}). The third item follows from the fact that $a_w^v$ is obtained as a $q$-weighted count of the character of the global Weyl modules, whose set of $H$-weights are contained in $\Sigma ( \la )$ (see e.g. \cite[\S 1.2]{Kat18}).

We prove the fourth item. The open dense $G[\![z]\!]$-orbit $\bO$ of $\bQ_G ( e )$ is the affine fibration over $\sB$, and its fiber is a homogeneous space of $\ker \, ( G[\![z]\!] \to G )$. Since the restriction from $\bQ_G ( e )$ to $\sB$ passes $\C_\mu \otimes \cO_{\bQ_G ( e )} ( \la )$ to $\C_\mu \otimes \cO_{\sB} ( \la )$ ($\la, \mu \in \La$), this restriction yields a $\C [\bH]$-linear map
$$K _{\bH} ( \bQ_G ( e ) ) \longrightarrow K_{\bH} ( \bO ) \stackrel{\cong}{\longrightarrow} K_{\bH} ( \sB ),$$
with its kernel spanned by $[\cO_{\bQ (  u t_{\beta} )}]$ for $u \in W$ and $\beta \neq 0$. This also maps $[\cO_{\bQ ( u )}]$ to $[\cO_{\sB ( u )}]$ for each $u \in W$. Since $v \not\in W$ and $v \le_\si e$ implies $v = u t_{\beta}$ with $u \in W$ and $0 \neq \beta \in Q^{\vee}_+$, we conclude the assertion in the third item.
\end{proof}

\begin{lem}[\cite{Kat18c} Lemma 1.14]\label{div}
For each $i \in \tI$, we have
$$[\cO_{\bQ_G ( s_i )}] = [\cO_{\bQ_G ( e )}] - e^{\varpi_i} [\cO_{\bQ_G ( e )} ( - \varpi_{i} )].$$
\end{lem}

We consider a $\C [\bH]$-module endomorphism $H_i$ ($i \in \tI$) of $K_{\bH} ( \bQ_G^{\mathrm{rat}} )$ as:
$$H_i : [\cO_{\bQ_G ( w )}] \mapsto [\cO_{\bQ_G ( w )}] - e ^{\varpi_i} [\cO_{\bQ_G ( w )} ( - \varpi_i )] \hskip 5mm w \in W_\af.$$

\begin{lem}\label{csq}
For $i, j \in \tI$, we have
$$\Xi ( \varpi_i ) \circ Q^{\al^{\vee}_j} = q^{- \left< \al^{\vee}_j, \varpi_i \right>} Q^{\al^{\vee}_j} \circ \Xi ( \varpi_i ) \in \mathrm{End}_{\sH_q} K _\bH ( \bQ_G^{\mathrm{rat}} ).$$
\end{lem}

\begin{proof}
For each $w \in W_\af$, we have
\begin{align*}
\Xi ( \varpi _i ) ( [\cO_{\bQ_G ( w )}] ) = & \sum_{v \in W_\af} a_v^w [\cO_{\bQ_G ( v )}], \hskip 5mm \text{where} \hskip 3mm a_v^w \in \C [\bH] \text{ and}\\
\gch \, \Gamma ( \bQ_G ( w ), & \cO_{\bQ_G ( w )} ( \la + \varpi_i ) ) = \sum_{v \in W_\af} a_v^w \gch \, \Gamma ( \bQ_G ( v ), \cO_{\bQ_G ( v )} ( \la ) )
\end{align*}
for each $\la \in \La_+$. Since we have
$$\gch \, \Gamma ( \bQ_G ( w t_{\gamma} ), \cO_{\bQ_G ( w t_{\gamma} )} ( \la ) ) = q^{- \left< \gamma, \la \right>} \gch \, \Gamma ( \bQ_G ( w ), \cO_{\bQ_G ( w )} ( \la ) )$$
for each $\gamma \in Q^{\vee}$ and $\la \in \La$ by \cite[Corollary A.4]{Kat18d}, we deduce that
$$\Xi ( \varpi_i ) \circ Q^{\al^{\vee}_j} ( [\cO_{\bQ_G ( w )}] ) = q^{- \left< \al^{\vee}_j, \varpi_i \right>} Q^{\al^{\vee}_j} \circ \Xi ( \varpi_i ) ( [\cO_{\bQ_G ( w )}] ).$$
Thus, the $\C [\bH]$-linearity of the composition maps implies the result.
\end{proof}

The following result is a version of the Demazure character formula for semi-infinite flag manifolds \cite[Theorem A]{Kat18}:

\begin{thm}\label{dcf}
Let $w \in W$ and $\la \in \La$. We have
$$D_{t_{w \beta}} [ \cO_{\bQ_G ( w )} ( \la ) ] = [ \cO_{\bQ_G ( wt_{\beta} )} ( \la ) ] = q^{- \left< \beta, \la \right>} Q^{\beta} [ \cO_{\bQ_G ( w )} ( \la ) ]$$
for every $\beta \in Q^{\vee}_<$. Moreover, $\{ \cO_{\bQ_G ( e )} ( \la )\}_{\la \in \La}$ is a $\C_q (\!(Q^{\vee})\!)$-free basis of $K_{\bG} ( \bQ^\ra_G )$.
\end{thm}

\begin{proof}
The first assertion for $\la \in \La_+$ is \cite[Theorem 4.13]{Kat18} (it lifts to the formal version by \cite{KNS17}). In view of Theorem \ref{HQc}, it prolongs to all $\la \in \La$. This proves the first assertion.

We prove the second assertion. Note that $\bigoplus_{u \in W} \C [\bH] [\cO_{\bQ_G ( u )}] \subset K_{\bH} ( \bQ^\ra_G )$ is stable by the $\sH_q ( \tI )$-action, and it is isomorphic to $K_{\bH} ( \sB )$ as $\sH_q ( \tI )$-modules by the comparison of the actions. In view of Theorem \ref{HQc} 2) and 4), it follows that the coefficient of $[\cO_{\bQ_G ( e )}]$ distinguishes two elements in the $D_{w_0}$-invariants of $\bigoplus_{u \in W} \C [\bH] [\cO_{\bQ_G ( u )}]$. Since we allow formal sums with respect to $Q^{\vee}_+$, we conclude that $\{ \cO_{\bQ_G ( e )} ( \la )\}_{\la \in \La}$ defines a $\C_q [\![Q^{\vee}]\!]$-free basis of $K_{\bG} ( \bQ_G ( e ) )$. Now the assertion follows by the $Q^{\vee}$-translations.\end{proof}

\begin{lem}\label{Lei2}
For each $i \in \tI_\af$, $\la \in \bX^*$, and $w \in W_\af$, we have
$$D_i ( e^{\la} [\cO_{\bQ_G ( w )}] ) \equiv \begin{cases}
e^{\la} [\cO_{\bQ_G ( w )}] + e^{s_i \la} [\cO_{\bQ_G ( s_i w )}] & \left< \al_i^{\vee}, \la \right> < 0, s_i w >_\si w \\
e^{\la} [\cO_{\bQ_G ( s_i w )}] & \left< \al_i^{\vee}, \la \right> = 0, s_i w >_\si w \\
- e^{s_i\la} [\cO_{\bQ_G ( w )}] + e^{s_i \la} [\cO_{\bQ_G ( s_i w )}] & \left< \al_i^{\vee}, \la \right> > 0, s_i w >_\si w \\
( e^{\la} + e^{s_i \la} ) [\cO_{\bQ_G ( w )}] & \left< \al_i^{\vee}, \la \right> < 0, s_i w <_\si w \\
e^{\la} [\cO_{\bQ_G ( w )}] & \left< \al_i^{\vee}, \la \right> = 0, s_i w <_\si w \\
0 & \left< \al_i^{\vee}, \la \right> > 0, s_i w <_\si w
\end{cases}$$
modulo the $\C_q$-span of $\{ e^{\mu} [\cO_{\bQ_G ( v )}] \mid \mu \in \Sigma_* ( \la ), v \in W_\af \}$.
\end{lem}

\begin{proof}
The assertion follows from the behavior of the Hecke operators (i.e. $D_i - 1$) seen in (the $t = 0$ version of the $t^{1/2}$-twist of) \cite[Proposition 3.3]{Che95}. One can also directly prove using Corollary \ref{Lei} and the convexity results in \cite[\S 1]{Che95}.
\end{proof}

Let $\la \in \La$. We consider two subspaces
\begin{align*}
K _{\preceq \la} := & \, \Span_{\C_q} \{ e^{\mu} [\cO_{\bQ_G ( w )}] \mid w \in W_\af, \mu \in \Sigma ( \la )\} \subset K_{\bH} ( \bQ_G^\ra )\\
K _{\prec \la} := & \, \Span_{\C_q} \{ e^{\mu} [\cO_{\bQ_G ( w )}] \mid w \in W_\af, \mu \in \Sigma_* ( \la )\} \subset K_{\bH} ( \bQ_G^\ra ).
\end{align*}
Here we stress that our span consists of finite sums.

\begin{cor}\label{Kla}
For each $\la \in \La$, the spaces $K _{\prec \la} \subset K _{\preceq \la}$ are $\sH_q^0$-submodules of $K_{\bH} ( \bQ_G^\ra )$.
\end{cor}

\begin{proof}
Combine Theorem \ref{H-si}, Corollary \ref{Lei}, and Lemma \ref{Lei2}.
\end{proof}

\begin{thm}\label{X-weight}
For each $\la \in \La$, we have a unique element $C ( \la ) \in K _{\preceq \la}$ with the following properties:
\begin{enumerate}
\item We have $C ( \la ) \equiv D_{w_0} ( e^{w_0 \la} [\cO_{\bQ_G ( w_0 )}] )\mod K_{\prec \la}$;
\item For each $\beta \in Q^{\vee}_<$, we have $D_{t_{\beta}} C ( \la ) = q^{- \left< \beta, \la \right>} C ( \la ) Q^\beta$.
\end{enumerate}
\end{thm}

\begin{proof}[Proof of Theorem \ref{X-weight}]
We prove the assertion by induction on the inclusion relation between $\Sigma ( \la )$. We assume that $D_{w_0} K_{\prec \la}$ is spanned by the joint eigenvectors with respect to the action of $\{ D_{t_{\beta}} \}_{\beta \in Q^{\vee}_<}$, and construct $C ( \la ) \in D_{w_0} K_{\preceq \la}$. Thanks to Theorem \ref{dcf} and Theorem \ref{HQc}, the element $C ( \la )$ exists (in fact uniquely) as an element in $K_{\bH} ( \bQ ^\ra_G )$.

The case $\la = 0$ is clear by setting $C ( 0 ) := D_{w_0} ( [\cO_{\bQ_G ( w_0 )}] ) = [\cO_{\bQ_G ( e )}]$ thanks to Lemma \ref{Lei2}.

We consider the general case by induction. Write $e >_\si w = u t_{\gamma}$ for $u \in W$ and $\gamma \in Q^{\vee}_+$. Let $\beta' \in Q^{\vee}_<$ be such that $\gamma + \beta' \in Q^{\vee}_{<}$. We have
$$\ell ( w t_{\beta'} ) = \ell ( t_{\beta'} ) - \ell ( u ) - 2 \left< \gamma, \rho \right> \hskip 5mm \text{and hence} \hskip 5mm \ell ( w t_{\beta'} ) < \ell ( t_{\beta'} )$$
by Theorem \ref{af-len}. It follows that
$$\ell ( t_{\beta + \beta'} ) > \ell ( w t_{\beta'} ) + \ell ( t_{\beta} ) \hskip 10mm \beta \in Q^{\vee}_<.$$
Consequently, the coefficient of $[\cO_{\bQ_G ( t_{\beta} )}]$ of $D_{t_{\beta}} ( C ( \la ) )$ modulo $K_{\prec \la}$ must be determined by the coefficient of $[\cO_{\bQ_G ( e )}]$ in $C ( \la )$ by Lemma \ref{Lei2}, that is $e^{t_{\beta} ( \la )} = q^{- \left< \beta, \la \right>} e^{\la}$. We set
$$C' ( \la ) := D_{w_0} ( e^{w_0 \la} [\cO_{\bQ_G ( w_0 )}] ).$$
Since $D_{t_{\beta}} ( C' ( \la ) )$ is $D_{w_0}$-invariant, we conclude that
$$D_{t_{\beta}} ( C' ( \la ) ) = q^{- \left< \beta, \la \right>} C' ( \la ) Q^{\beta} \mod K_{\prec \la}$$
by Theorem \ref{dcf}. In particular, we find that
\begin{equation}
D_{t_{\beta}} ( C' ( \la ) ) - q^{- \left< \beta, \la \right>} C' ( \la ) Q^{\beta} \in K_{\prec \la}.\label{Cdiff}
\end{equation}
By the first condition of our assertion and the induction hypothesis, we find that $D_{w_0} K_{\prec \la}$ is spanned by $\{ C ( \mu ) \}_{\mu \in \Sigma_* ( \la )}$ as a $\C_q Q^{\vee}$-module. These are the $D_{t_{\beta}}$-eigenvectors for each $\beta \in Q^{\vee}_<$. We expand the LHS of (\ref{Cdiff}) as
$$\sum_{\mu \in \Sigma_* ( \la )} C ( \mu ) b_{\la}^{\mu} \hskip 5mm b_{\la}^{\mu} \in \C_q Q^{\vee}_+.$$
Here we remark that this sum must be finite.

For any choices of $c_{\la}^{\mu} \in \C ( q ) [\![Q^{\vee}_+]\!]$ ($\mu \in \La$), we have
\begin{align*}
D_{t_{\beta}} ( C' ( \la ) - & \, \sum_{\mu \in \Sigma_* ( \la )} C ( \mu ) c_{\la}^{\mu} ) - q^{- \left< \beta, \la \right>} ( C' ( \la ) - \sum_{\mu \in \Sigma_* ( \la )} C ( \mu ) c_{\la}^{\mu} )\\
& = \sum_{\mu \in \Sigma_* ( \la )} C ( \mu ) ( b_{\la}^{\mu}  - q^{- \left< \beta, \mu \right>} c_\la^\mu + q^{- \left< \beta, \la \right>} c_\la^\mu ).
\end{align*}
It follows that the element
\begin{equation}
C' ( \la ) - \sum_{\mu \in \Sigma_* ( \la )} c_{\la}^{\mu} C ( \mu ) \hskip 5mm c_{\la}^{\mu} := \frac{q^{\left< \beta, \mu \right>}}{1 - q^{\left< \beta, \mu - \la \right>}} b_{\la}^{\mu} \in \frac{1}{1 - q^{\left< \beta, \mu - \la \right>}}\C_q Q^{\vee}_+\label{eigC}
\end{equation}
satisfies the desired properties in $\C ( q ) \otimes_{\C_q} K _{\preceq \la}$ (note that we have $\left< \beta, \mu - \la \right> \neq 0$ for every $\mu \in \Sigma_* ( \la )$ for some choice of $\beta$). Here we remark that the coefficients $\{ c_{\la}^{\mu} \}_{\mu}$ does not depend on the choice of $\beta \in Q^{\vee}_<$ by the characterization in $\C ( q )\otimes_{\C_q} K_{\bH} ( \bQ^\ra_G )$ coming from Theorem \ref{dcf}. Thus, we conclude that (\ref{eigC}) belongs to
$$K _{\preceq \la} = \left( C ( q ) \otimes_{\C_q} K _{\preceq \la}  \right) \cap K_{\bH} ( \bQ^\ra_G )\subset \C ( q )\otimes_{\C_q} K_{\bH} ( \bQ^\ra_G ).$$

Therefore, we obtain the desired element $C ( \la )$ inside $K_{\preceq \la}$ by induction. Hence, the induction proceeds and we conclude the result.
\end{proof}

\begin{cor}\label{fundC}
For each $i \in \tI$, we have
$$[\cO_{\bQ_G ( e )} ( \varpi_i )] = C ( \varpi_i ) \frac{1}{1 - Q^{\al_i^{\vee}}} := \sum_{m \ge 0} C ( \varpi_i ) Q^{m \al_i^{\vee}}.$$
\end{cor}

\begin{proof}
Compare $C ( \varpi_i )$ with the Pieri-Chevalley rule in \cite[Theorem 5.10]{KNS17} through Theorem \ref{dcf}.
\end{proof}

\begin{thm}[\cite{Kat18c} Theorem 3.11 and Remark 3.12]\label{qK-int}
There exists a $R ( \bG )$-linear embedding
$$\Psi_G : qK_{\bG} ( \sB ) \hookrightarrow K_{\bG} ( \bQ_G^{\ra} )$$
such that:
\begin{enumerate}
\item $\Psi_{G} ( Q^{\beta} ) = [\cO_{\bQ_G ( t_{\beta} )}]$ for each $\beta \in Q^{\vee}_+$;
\item $\Psi_G ( A^{\la} (\bullet) ) = \Xi ( \la ) ( \Psi_G ( \bullet ) )$ for each $- \la \in \La_+$. \hfill $\Box$
\end{enumerate}
\end{thm}

\section{Darboux coordinates of $K_\bG ( \Gr_G )_{\lo}$}

We work in the same settings as in \S \ref{setup}.

\subsection{Non-commutative $K$-theoretic Peterson isomorphism}

\begin{thm}\label{K-id}
Assume that $G$ is simple. We have a $\sH_q^\sh$-module embedding
$$\Phi_G : K _{\bG} ( \Gr _G )_{\lo} \hookrightarrow K_{\bG} ( \bQ_G^{\ra} )$$
that sends $[\cO_{\Gr_G ( 0 )}]$ to $[\cO_{\bQ_G ( e )}]$, intertwines the right product $\odot_q$ on the LHS to the tensor product on the RHS. More precisely, we have: For each $i \in \tI$ and $\xi \in K _{\bG} ( \Gr _G )$, it holds
$$\Phi ( \xi \odot_q ( e^{- \varpi_i} - e^{- \varpi_i} \bh_i ) ) = \Xi ( - \varpi_i ) ( \xi ).$$
\end{thm}

To prove Theorem \ref{K-id}, we need:

\begin{lem}\label{e-char}
We have an isomorphism
$$\End_{\sH_q^\sh} ( K_{\bG} ( \Gr_G )_\lo ) \cong K_{\bG} ( \Gr_G )_\lo$$
determined by the image of $[\cO_{\Gr_G ( 0 )}]$. In particular, every $\sH_q^\sh$-endomorphism of $K_{\bG} ( \Gr_G )_\lo$ is obtained by the composition of the right multiplication of $K_{\bG} ( \Gr_G )$ followed by the application of $\mt_{\gamma}$ for some $\gamma \in \bX_*$.
\end{lem}

\begin{proof}
As the torus factor $H'$ of $G$ produces $K_{\bH'} ( \Gr_{H'} ) = K_{\bH'} ( \Gr_{H'} )_\lo$ as a ($\C_q$-)tensor factors of $K_{\bG} ( \Gr_G )$ and $K_{\bG} ( \Gr_G )_\lo$ that are isomorphic to a Heisenberg algebra, we can factor out such a factor to assume that $G$ is simple.

Since $K_{\bG} ( \Gr_G )$ affords a regular representation of $\sH_q^\sh$, we see that
$$\End_{\sH_q^\sh} ( K_{\bG} ( \Gr_G ) ) \cong K_{\bG} ( \Gr_G ).$$
Here the isomorphism is obtained by the right multiplication and hence $f \in \End_{\sH_q^\sh} ( K_{\bG} ( \Gr_G ) )$ is determined by $f ( 1 )$.

Let $f \in \End_{\sH_q^\sh} ( K_{\bG} ( \Gr_G ) )$. By construction of $K_{\bG} ( \Gr_G )_\lo$, we can take $\beta \in \bX_*^{-}$ such that $f ( 1 ) \odot_q \mt_{\beta} \in K_{\bG} ( \Gr_G )$. It follows that $1 \mapsto f ( 1 ) \odot_q \mt_{\beta}$ uniquely gives rise to an element of $\End_{\sH_q^\sh} ( K_{\bG} ( \Gr_G ) )$. Since the right action of $\mt_{\beta}$ is invertible, we conclude that $f ( 1 ) \in K_{\bG} ( \Gr_G )_\lo$ already defines an element of $\End_{\sH_q^\sh} ( K_{\bG} ( \Gr_G )_\lo )$ uniquely as required.
\end{proof}

\begin{proof}[Proof of Theorem \ref{K-id}]
Thanks to \cite[Proposition 2.13 and Remark 2.14]{Kat18c}, we have a $\sH_q^\sh$-module embedding
$$\Phi_G : K _{\bG} ( \Gr _G )_{\lo} \hookrightarrow K_{\bG} ( \bQ_G^{\ra} )$$
that sends $\mt_{\beta}$ to $[\cO_{\bQ_G ( t_{\beta} )}]$ as the (left) $D_{w_0}$-invariant part of the corresponding embedding of $\bH$-equivariant $K$-groups (cf. Corollary \ref{ncGr}).

From the construction of the map $\Phi_G$ through its $\bH$-equivariant variants, we see that $K_{\bG} ( \bQ_G^{\ra} )$ is the completion of the image of $\Phi_G$ with respect to the topology given in \S \ref{sif}. In view of Lemma \ref{e-char}, we find that $\Xi ( \la )$ defines an element of $\End_{\sH_q^\sh} ( K_{\bG} ( \Gr_G )_\lo )$ if and only if $\Xi ( \la ) ( [\cO_{\bQ_{G} ( e )}] )$ is a finite linear combination of $\{ [\cO_{\bQ_G ( w )}]\}_{w \in W_\af}$. This happens for $\la = - \varpi_i$ by Lemma \ref{div}. Namely, we have $\Xi ( - \varpi_i ) = e ^{- \varpi_i} ( \mathrm{id} - H_i )$. Again by \cite[Proposition 2.13 and Remark 2.14]{Kat18c}, we conclude that $\Xi ( - \varpi_i )$ induces a(n left $\sH_q^\sh$-module) endomorphism of $K _{\bG} ( \Gr _G )_{\lo}$ that sends $[\cO_{\Gr_G ( 0 )}]$ to $e ^{-\varpi_i} ( \mathrm{id} - \bh_i )$. Therefore, we conclude that the equality in the assertion.
\end{proof}

\begin{cor}\label{ncGr}
Assume that $G$ is simple. We have a $\sH_q$-module embedding
$$\Phi : K _{\bH} ( \Gr _G )_{\lo} \hookrightarrow K_{\bH} ( \bQ_G^{\ra} )$$
extending $\Phi_G$ with the following properties:
\begin{enumerate}
\item we have $\Phi ( [\cO_{\Gr_G ( u t_{\beta} )}] ) = [\cO_{\bQ_G ( u t_{\beta} )}]$ for $u \in W$ and $\beta \in Q^{\vee}_<$;
\item the right multiplication by $\mt_{\gamma}$ corresponds to the right translation by $\gamma \in Q^{\gamma}$ for each $\gamma \in Q^{\vee}$;
\item For each $i \in \tI$ and $\xi \in K _{\bH} ( \Gr _G )_{\lo}$, it holds
$$\Phi ( \xi \odot_q \bh_i ) = H_i ( \xi ).$$
\end{enumerate}
\end{cor}

\begin{proof}
Notice that we have $[\cO_{\sB}] \in K_\bG ( \sB )$ in Theorem \ref{finK}, that results in $\sH_q ( \tI ) K_\bG ( \sB ) = K_\bH ( \sB )$ by Theorem \ref{finK} 5). The comparison of Theorem \ref{finK} with Theorem \ref{sH-reg} yields
$$\sH_q K _{\bG} ( \Gr _G )_{\lo} = \C_q [H] K _{\bG} ( \Gr _G )_{\lo} = K _{\bH} ( \Gr _G )_{\lo},$$
while the comparison of Theorem \ref{finK} with Theorem \ref{H-si} yields
$$\sH_q K _{\bG} ( \bQ_G^{\ra} ) = \C_q [H] K _{\bG} ( \bQ_G^{\ra} ) = K _{\bH} ( \bQ_G^{\ra} )$$
as $\sH_q$-modules with the desired properties except for the first item. The first item follows from \cite[Proposition 2.13 and Remark 2.14]{Kat18c}.
\end{proof}

\begin{cor}\label{c-char}
Keep the setting of Lemma \ref{e-char}. Each $\sH_q^\sh$-module endomorphism of $K_{\bG} ( \Gr_G )_\lo$ is continuous with respect to the the topology induced from the topology of $K_{\bH} ( \bQ_{[G,G]}^\ra )$ $($defined in $\S \ref{sif})$ under $\Phi_{[G,G]}$ $($by extending the scalar from $\C_q$ to $K_{\bH'} ( \Gr_{H'} ))$. \hfill $\Box$
\end{cor}

\subsection{Darboux generators of $K_{\bG} ( \Gr_G )_\lo$}

For each $i \in \tI$, we set
$$\phi_i := e ^{- \varpi_i} ( \mathrm{id} - \odot_q \bh_i ) \in K_{\bG} ( \Gr_G )_\lo \cong \End_{\sH_q^\sh} ( K_{\bG} ( \Gr_G )_\lo ).$$

\begin{lem}\label{Heis-comp}
Assume that $G$ is simple. There exists a unique $\sH_q^\sh$-module endomorphism $\xi_i$ on $K_{\bG} ( \Gr_G )_\lo$ for each $i \in \tI$ such that
$$\xi_i \circ \phi_i = ( \mathrm{id} - \mt_{\al^{\vee}_i} ) \hskip 3mm \text{and} \hskip 3mm \phi_i \circ \xi_i = ( \mathrm{id} - q \mt_{\al^{\vee}_i} ).$$
In addition, we have
$$\xi_i \circ \xi_j = \xi_j \circ \xi_i, \hskip 3mm \xi_i \circ \phi_j = \phi_j \circ \xi_i, \hskip 3mm \text{and} \hskip 3mm \phi_i \circ \phi_j = \phi_j \circ \phi_i \hskip 5mm i \neq j.$$
\end{lem}

\begin{proof}
We transplant these endomorphisms to $K_{\bG} ( \bQ_G^\ra )$. We set $\phi_i$ to be the endomorphism $\Xi ( - \varpi_i )$, and $\xi_i$ to be the endomorphism $( 1 - Q^{\al_i^{\vee}} ) \Xi ( \varpi_i )$ for each $i \in \tI$. A priori, $\xi_i$ only defines a $\sH_q^\sh$-endomorphism of the completion of $K_{\bG} ( \Gr_G )_\lo$ that is isomorphic to $K_{\bG} ( \bQ^{\ra} )$ via (the natural extension of) $\Phi_G$.  To see that $\xi_i$ defines an endomorphism of $K_{\bG} ( \Gr_G )_\lo$, it suffices to see that whether $( 1 - Q^{\al_i^{\vee}} ) \Xi ( \varpi_i )$ defines an endomorphism of $K_{\bG} ( \Gr_G )_\lo$. By Corollary \ref{c-char}, it suffices to see that
$$( 1 - Q^{\al_i^{\vee}} ) \Xi ( \varpi_i ) ( [\cO_{\bQ_G ( e ) }]) = [\cO_{\bQ_G ( e )} ( \varpi_i)] - [\cO_{\bQ_G ( t_{\al_i^{\vee}} )} ( \varpi_i)]$$
is a finite linear combination of $\{ [\cO_{\bQ_G ( w ) }] \}_{w \in W_\af}$, that is the content of Corollary \ref{fundC}. Now the commutation relation between them follow from Lemma \ref{csq}.
\end{proof}

\begin{cor}
Keep the setting of Lemma \ref{Heis-comp}. Then, the elements
\begin{equation}
\Phi_G ( \left( \prod_{i \in \tI, \left< \al_i^{\vee}, \la \right> <0} \xi_i^{- \left< \al_i^{\vee}, \la \right>} \right) \left( \prod_{i \in \tI, \left< \al_i^{\vee}, \la \right> >0} \phi_i^{\left< \al_i^{\vee}, \la \right>} \right) [\cO_{\Gr_G ( 0 )}] ) \hskip 5mm \la \in \La\label{ClaG}
\end{equation}
are $\C_q Q^{\vee}$-linearly independent in $K_{\bG} ( \bQ^\ra )$. In particular, there is no additional relations among $\{ \xi_i, \phi_i \}_{i \in \tI}$ $($to those presented in Lemma $\ref{Heis-comp})$.
\end{cor}

\begin{proof}
The elements in (\ref{ClaG}) are non-zero since $\phi_i$ and $\xi_i$ defines $\Xi ( - \varpi_i )$ and $( 1 - Q^{\al_i^{\vee}} ) \Xi ( \varpi_i )$ for each $i \in \tI$, that are invertible in $K_{\bG} ( \bQ^\ra )$. In view of Theorem \ref{dcf}, these elements belong to different (joint) eigenspaces with respect to the action of $D_{t_{\beta}}$ ($\beta \in Q^{\vee}_<$), and hence they are $\C_q Q^{\vee}$-linearly independent. If we have an additional relation among $\{ \xi_i, \phi_i \}_{i \in \tI}$, then it violates the linear independence of (\ref{ClaG}). Consequently, it is impossible and hence the relations presented in Lemma \ref{Heis-comp} is optimal.
\end{proof}

We set $qK_{\bH} ( \sB )_\lo := \C Q^{\vee} \otimes_{\C Q^{\vee}_+} qK_{\bH} ( \sB )$.

\begin{thm}\label{ncPet}
Assume that $G$ is simple. We have a $\sH_q$-module isomorphism
$$\Psi^{-1} \circ \Phi : K _{\bH} ( \Gr _G )_{\lo} \hookrightarrow qK_{\bH} ( \sB )_\lo$$
with the following properties:
\begin{enumerate}
\item We have $(\Psi^{-1} \circ \Phi) ( [\cO_{\Gr_G ( u )}] \mt_{\beta} ) = [\cO_{\sB ( u )}]Q^{\beta}$ for $u \in W$ and $\beta \in Q^{\vee}$;
\item For each $i \in \tI$ and $\xi \in K _{\bG} ( \Gr _G )_{\lo}$, it holds
$$(\Psi^{-1} \circ \Phi) ( \phi_i ( \xi ) ) = A^{- \varpi_i} \left( (\Psi^{-1} \circ \Phi) ( \xi ) \right).$$
\end{enumerate}
\end{thm}

\begin{proof}
The existence of the isomorphism and the first item follows from Corollary \ref{ncGr} and \cite[Theorem 4.1 and its proof]{Kat18c}. The second item is a consequence of the identification of $\phi_i$ with $\Xi ( - \varpi_i )$ under $\Phi$.
\end{proof}

\begin{prop}\label{resGH}
We have a $\C_q$-algebra embedding
$$K_{\bG} ( \Gr_G )_\lo \hookrightarrow K_{\bH} ( \Gr_H )$$
given by $\mt_{\gamma} \mapsto \mt_{\gamma}$ $(\gamma \in \bX_*)$, $e^{\la} \mapsto e^{\la}$ $(\la \in \bX^* ( G ) )$, and
$$\phi_i \mapsto e^{- \varpi_i}, \xi_i \mapsto ( 1 - \mt_{\al_i^{\vee}} ) e^{\varpi_i} \hskip 2mm (i \in \tI).$$
\end{prop}

\begin{rem}
{\bf 1)} Taking Theorem \ref{K-id} into account, Proposition \ref{resGH} follows as the symmetrization of a result of Daniel Orr \cite[(0.2) and Theorem 5.1]{Orr20} when $G$ is simple of types $\mathsf{ADE}$; {\bf 2)} By taking the $q = 1$ specialization, this embedding becomes an embedding of commutative algebras that gives rise to an isomorphism between their fraction fields.
\end{rem}

\begin{proof}[Proof of Proposition \ref{resGH}]
The element $e^{\la}$ $(\la \in \bX^* ( G ))$ and $t_{\gamma}$ ($\gamma \in \bX_* ( G )$) generates a common subalgebras of the both sides. If we add these elements to the case of $G =[G,G]$, then we obtain the whole embedding. Thus, we can assume that $G$ is simple. 
 
The commutation relation is preserved by a direct calculation. Thus, it remains to see that the elements in Proposition \ref{resGH} generates the whole $K_{\bG} ( \Gr_G )_\lo$. We have
\begin{align*}
\left( \prod_{j = 0}^{m-1} ( 1 - q^{-j} Q^{\al_i^{\vee}}) \right) \Xi ( m \varpi_i ) & = \left( \prod_{j = 0}^{m-1} ( 1 - q^{-j} Q^{\al_i^{\vee}}) \right) \Xi ( \varpi_i )^m\\
= ( 1 - Q^{\al_i^{\vee}}) \Xi ( \varpi_i ) & \left( \prod_{j = 0}^{m-2} ( 1 - q^{-j} Q^{\al_i^{\vee}}) \right) \Xi ( \varpi_i )^{m-1}\\
& = \cdots\\
& = \left( ( 1 - Q^{\al_i^{\vee}}) \Xi ( \varpi_i ) \right)^m.
\end{align*}

The Pieri-Chevalley rule \cite[Theorem 5.13]{KNS17} is $\C [\bH]$-linear, and the action of $\Xi ( \varpi_i )$ sends the Schubert class $[\cO_{\bQ ( w )}]$ ($w \in W_\af$) to a possibly infinite sum
$$e^{\mu} [\cO_{\bQ ( v )}] \hskip 5mm w \ge_\si v \in W_\af, \mu \in \Sigma ( \varpi_i ).$$
In view of Corollary \ref{fundC}, the action of $( 1 - Q^{\al_i^{\vee}})\Xi ( \varpi_i )$ sends the Schubert class $\cO_{\bQ ( e )}$ to a linear combination of
$$e^{v \varpi_i} [\cO_{\bQ ( v )}] \hskip 5mm v \in W$$
modulo the formal sum of $e^{\mu} [\cO_{\bQ ( v )}]$ for $\mu \in \Sigma_* ( \varpi_i )$ and $v \in W_\af$. In addition, the term of the shape $e^{\varpi_i} [\cO_{\bQ ( v )}]$ must be $e^{\varpi_i} [\cO_{\bQ ( e )}]$ by inspection (using Lemma \ref{Lei2}).

We have $[Q^{\al^{\vee}_i}, \Xi ( \pm \varpi_j )] = 0$ for $i \neq j$ (Lemma \ref{csq}). In view of Theorem \ref{dcf} and the fact that $Q^{\beta}$ ($\beta \in Q^{\vee}$) commutes with the $\sH_q$-action, we deduce that
\begin{equation}
\left( \prod_{i \in \tI, \left< \al_i^{\vee}, \la \right> < 0} \Xi ( - \varpi_i ) ^{- \left< \al_i^{\vee}, \la \right>} \right) \prod_{i \in \tI, \left< \al_i^{\vee}, \la \right> > 0} \left( ( 1 - Q^{\al_i^{\vee}}) \Xi ( \varpi_i ) \right)^{\left< \al_i^{\vee}, \la \right>}[\cO_{\bQ ( e )}]\label{Cpres}
\end{equation}
is a (joint) eigenfunctions of $D_{t_{\gamma}}$ ($\gamma \in Q^{\vee}_<$). By Theorem \ref{X-weight}, we deduce that the $\C _q$-coefficient of the term $e^{\mu}[\cO_{\bQ ( w )}]$ ($w \in W_\af$) in (\ref{Cpres}) is non-zero only if $\mu \in \Sigma ( \la )$, and the class (\ref{Cpres}) is uniquely determined by the $\C_q$-coefficients of $e^{\la}[\cO_{\bQ ( t_{\beta} )}]$ for all $\beta \in Q^{\vee}$.

We first examine the case $\la \in \La_+$. Since $\la \in \Sigma ( \la )$ is an extremal point, we find that $( \la + \varpi_i) \in \Sigma ( \la + \varpi_i )$ is attained uniquely as the sum of elements from $\Sigma ( \la )$ and $\Sigma ( \varpi_i )$ whenever $\la \in \La_+$ (namely the sum of $\la \in \Sigma ( \la )$ and $\varpi_i \in \Sigma ( \varpi_i )$). From this, we find that the $\C _q$-coefficient of the term $e^{\la}[\cO_{\bQ ( w )}]$ ($w \in W_\af$) is just one for $w = e$ and it is zero for $w \neq e$ by induction from the case $\la = 0 \in \La_+$. Since the both sides are (joint) eigenfunctions of $D_{t_{\gamma}}$ ($\gamma \in Q^{\vee}_<$) with common (joint) eigenvalues whose coefficients of $e^{\la}[\cO_{\bQ ( t_{\beta} )}]$ ($\beta \in Q^{\vee}$) are the same, we conclude
\begin{align*}
C_{\la} = \left( \prod_{i \in \tI} ( ( 1 - Q^{\al_i^{\vee}}) \Xi ( \varpi_i ) )^{\left< \al_i^{\vee}, \la \right>}\right) [\cO_{\bQ_G ( e )}] \hskip 5mm \la \in \La_+
\end{align*}
by Theorem \ref{X-weight}.

Now we consider general $\la \in \La$. Find $\tJ \subset \tI$, $\la_+ \in \La_+^{(\tI \setminus \tJ)}$, and $\la_- \in \La_+^\tJ$ such that $\la = \la_+ - \la_-$. When $\la_- = 0$, then the weight $e^{\la_+}$ appears only as a coefficient of $[\cO_{\bQ ( e )}]$ in $C_{\la_+}$ by the previous paragraph. If we want to represent $\la \in \La$ by a sum of elements from $\Sigma ( \la_+ )$ and $\Sigma ( - \la _- ) = \Sigma ( - w_0 ^{\tJ} \la_- )$, then we have necessarily $\la = \la_+ - \la_-$ since $\la$ belongs to the same $W$-orbit as $\la_+ - w_0 ^{\tJ} \la_- \in \La_+$. The coefficient of $e^{- \la_-}[\cO_{\bQ ( t_{\beta} ) }]$ in $C_{- \la_-}$ is one if $\beta = 0$, and zero if $\beta \neq 0$ by \cite[Corollary 3.15]{NOS18} (note that the set of paths $\mathrm{QLS} ( \la_- )$ contains a unique path whose weight is of the form $q^* e^{\la_-}$ since it represents the character of a local Weyl module, and such a path contributes to $[\cO_{\bQ ( e )}]$ only once by the shape of the formula). It follows that the coefficient of $e^{\la}[\cO_{\bQ ( t_{\beta} ) }]$ in $C_{\la}$ is one if $\beta = 0$, and zero if $\beta \neq 0$. Therefore, we conclude that (\ref{Cpres}) must be $C_{\la}$ for every $\la \in \La$.

It follows that
$$\Phi_G^{-1} ( C_\la ) = \left( \prod_{i \in \tI, \left< \al_i^{\vee}, \la \right> <0} \xi_i^{- \left< \al_i^{\vee}, \la \right>} \right) \left( \prod_{i \in \tI, \left< \al_i^{\vee}, \la \right> > 0} \phi_i^{\left< \al_i^{\vee}, \la \right>} \right) [\cO_{\Gr_G ( 0 )}] \in K_{\bG} ( \Gr_G )_\lo.$$
By Theorem \ref{X-weight} and Theorem \ref{K-id} (cf. Corollary \ref{ncGr}), one sees that $\{\Phi_G^{-1} ( C_\la ) \}_{\la \in P}$ forms a $\C_q Q^{\vee}$-basis of $K_{\bG} ( \Gr_G )_\lo$. Thus, the elements in the assertion generates the whole $K_{\bG} ( \Gr_G )_\lo$, and we have the desired inclusion.
\end{proof}

\begin{cor}\label{genKG}
The $\C_q$-algebra $K_{\bG} ( \Gr_G )_\lo$ is generated by $\mt_{\gamma}$ $(\gamma \in \bX_*)$, $e^{\la}$ $(\la \in \bX^* ( G ) )$, and $\phi_i, \xi_i$ $(i \in \tI)$. \hfill $\Box$
\end{cor}

\begin{cor}\label{absDarboux}
We have a $\C_q$-algebra embedding
$$K_{\bG} ( \Gr_G ) \hookrightarrow K_{\bH} ( \Gr_H )$$
obtained by the restriction of the domain in Proposition $\ref{resGH}$. \hfill $\Box$
\end{cor}

\section{Induction equivalence for flag varieties}

We work under the setting of \S \ref{sif}. In particular, $G$ is simple. The goal of this section is to present the following:

\begin{thm}\label{qK-surj}
Let $L = L ^\tJ$ be the standard Levi subgroup corresponding to $\tJ \subset \tI$. There is a $\C_q \bX^* ( G )$-linear surjective map
$$qK_{\bG} ( \sB )^{\wedge} \longrightarrow qK_{\bL} ( \sB^{\tJ} )^{\wedge}$$
sending $[\cO_{\sB}]$ to $[\cO_{\sB^{\tJ}}]$, and it intertwines the action of $A^{\pm \varpi_i}$ to the action of $A^{\pm \varpi_i}$ for each $i \in \tI$. In addition, the kernel of this map is generated by $Q^{-w_0 \al_i^{\vee}}$ for $i \in ( \tI \setminus \tJ )$.
\end{thm}

Theorem \ref{qK-surj} is proved in subsection \S \ref{sec:qKs} via explicit calculation. We record more general result as Theorem \ref{qK-surj-gen}.

\subsection{Reductions of quasi-map spaces}

\begin{lem}\label{Q-ind}
Let $\beta \in -w_0 Q^{\vee}_{\tJ, +}$. We have an isomorphism
$$\sQ_{G} ( \beta ) \cong G \times_{P ^\tJ} \sQ _{L ^\tJ} ( \beta ),$$
where the the unipotent radical of $P ^\tJ$ acts on $\sQ _{L ^\tJ} ( \beta )$ trivially.
\end{lem}

\begin{proof}
The definition of $\sQ _G ( \beta )$ is to consider a collection of $\C$-lines $\ell _\la$ in $V ( \la ) \otimes \C [z]$ for each $\la \in \La_+$ (cf. \cite[Lemma 3.28 and Theorem 3.30]{Kat18d}). In particular, such collections must satisfy the same relation as $\C (\!(z)\!)$-lines if we extend the scalar. By (\ref{Pemb}), we have $\ell_{\varpi_i} \in V ( \varpi_i ) \subset V ( \varpi_i ) \otimes \C (\!( z )\!)$ for $i \not\in \tJ$. Thanks to the Pl\"ucker relations (see e.g. \cite[Theorem 1.1.2]{BG}), we know that $\ell_{\varpi_i} \in G \bv_{\varpi_i}$ for $i \not\in \tJ$. Therefore, a point of $\sQ_{G} ( \beta )$ is $G$-conjugate to a point represented as a collection of $\C$-lines $\{\ell'_\la\}_{\la \in \La_+}$ such that $\ell'_{\varpi_i} = \C \bv_{\varpi_i}$ for $i \not\in \tJ$. By the Pl\"ucker relation (considered over the field $\C (\!(z)\!)$), it follows that $\ell'_{\varpi_j} \in L^\tJ (\!(z)\!) \bv_{\varpi_j}$ for $j \in \tJ$ in this case. This forces our point to belong to $\sQ _{L ^\tJ} ( \beta )$, with the trivial action of the unipotent radical of $P ^\tJ$. From these, we deduce a surjective homomorphism $G \times_{P ^ \tJ } \sQ _{L ^ \tJ } ( \beta ) \rightarrow \sQ _{G} ( \beta )$. Since the $G$-orbit of $\{ \C \bv_{\varpi_i} \}_{i \not\in \tJ}$ is $\sB_{\tJ}$, this map is a homeomorphism between projective normal varieties. It must be an isomorphism by the Zariski main theorem.  
\end{proof}

\begin{cor}\label{H-surj}
Keep the setting of Lemma \ref{Q-ind}. For each $\la \in \La_+$, we have a surjective $(P^\tJ$-module$)$ map
$$H^0 ( \sQ_{G} ( \beta ), \cO_{\sQ_{G} ( \beta )} ( \la ) ) \longrightarrow \!\!\!\!\! \rightarrow H^0 ( \sQ_{L ^ \tJ } ( \beta ), \cO_{\sQ_{L ^ \tJ } ( \beta )} ( \la ) ).$$
\end{cor}

\begin{proof}
In view of \cite[Theorem 3.33]{Kat18d}, we have a surjection
$$H^0 ( \bQ_{L ^ \tJ } ( e ), \cO_{\bQ_{L ^ \tJ } ( e )} ( \la ) ) \longrightarrow \!\!\!\!\! \rightarrow H^0 ( \sQ_{L ^ \tJ } ( \beta ), \cO_{\sQ_{L ^ \tJ } ( \beta )} ( \la ) ).$$
In view of \cite[Theorem 1.2]{Kat18d}, the $H$-weight of $H^0 ( \bQ_{L ^ \tJ } ( e ), \cO_{\bQ_{L ^ \tJ } ( e )} ( \la ) )$ is concentrated in $w_0 \la + Q^{\vee}_{\tJ,+}$. Since $\sQ_{L ^ \tJ } ( \beta )$ is stable under the $L ^ \tJ $-action, it follows that $H^0 ( \bQ_{L ^ \tJ } ( e ), \cO_{\bQ_{L ^ \tJ } ( e )} ( \la ) )$ is a direct sum of finite-dimensional irreducible $L ^ \tJ $-module. Since $\left< \al_i^{\vee}, \al_j \right> \le 0$ for every $i \in \tI \setminus \tJ$ and $j \in \tJ$ (and $\la \in \La_+$), every finite-dimensional irreducible $L ^ \tJ $-submodule in $H^0 ( \bQ_{L ^ \tJ } ( e ), \cO_{\bQ_{L ^ \tJ } ( e )} ( \la ) )$ is an irreducible $[L ^ \tJ , L ^ \tJ ]$-module twisted by a weight $\mu$ such that $\left< \al_i^{\vee}, \mu \right> \le 0$ for every $i \in ( \tI \setminus \tJ )$. It follows that
$$H^0 ( \sQ_{L ^ \tJ } ( \beta ), \cO_{\sQ_{L ^ \tJ } ( \beta )} ( \la ) )^* \hookrightarrow H^0 ( G / P ^ \tJ , \mathcal V )^*,$$
where $\mathcal V$ is the $G$-equivariant vector bundle obtained by inflating the $P ^ \tJ $-module $H^0 ( \sQ_{L ^ \tJ } ( \beta ), \cO_{\sQ_{L ^ \tJ } ( \beta )} ( \la ) )$. By the Leray spectral sequence, we have
$$H^0 ( G / P ^ \tJ , \mathcal V ) \cong H^0 ( \sQ_{G} ( \beta ), \cO_{\sQ_{G} ( \beta )} ( \la ) ).$$
Therefore, we conclude
$$H^0 ( \sQ_{G} ( \beta ), \cO_{\sQ_{G} ( \beta )} ( \la ) ) \cong H^0 ( G / P ^ \tJ , \mathcal V ) \longrightarrow \!\!\!\!\! \rightarrow H^0 ( \sQ_{L ^ \tJ } ( \beta ), \cO_{\sQ_{L ^ \tJ } ( \beta )} ( \la ) )$$
as desired.
\end{proof}

Let $\g [z] := \g \otimes \C [z]$ be the Lie algebra obtained by scalar extension. Each $\la \in \La_+$ defines a $\g [z]$-module $\bW _G ( \la )$ that is the global Weyl module in the sense of \cite{CP01}. By expressing $\la \in \La_+$ as the sum $\la = \la^{(1)} + \la^{(2)}$ of $\la^{(1)} \in \La_+^{\tJ}$ and $\la^{(2)} \in \La^{\tI \setminus \tJ}$, we have the corresponding global Weyl module $\bW _{[L ^ \tJ, L ^ \tJ]} ( \la^{(1)} )$ of $[\mathfrak{l}^{\tJ}, \mathfrak{l}^{\tJ}][z]$ (by taking the external tensor product of the global Weyl modules for all simple factors of $[L^{\tJ}, L^{\tJ}]$). We define
$$\bW _{L ^ \tJ} ( \la ) := \bW _{[L ^ \tJ, L ^ \tJ]} ( \la^{(1)} ) \otimes \C_{\la^{(2)}},$$
that is a $( [\mathfrak{l}^{\tJ}, \mathfrak{l}^{\tJ}][z] + \h )$-module.

\begin{cor}\label{inclWeyl}
For each $\la \in \La_+$, we have an inclusion $\bW _{L ^ \tJ} ( \la ) \subset \bW _G ( \la )$ between global Weyl modules.
\end{cor}

\begin{proof}
In view of \cite[Proposition 5.1]{BF14b} (cf. \cite[Theorem 3.33]{Kat18d}), we have
\begin{equation}
\bigcup_{\beta \in -w_0 Q^{\vee}_{\tJ,+}} H^0 ( \sQ_{L^\tJ} ( \beta ), \cO_{\sQ_{L^\tJ} ( \beta )} ( -w_0 \la ) )^* = \bW_{L ^ \tJ} ( \la ).\label{ind-W}
\end{equation}

By Corollary \ref{H-surj}, we have
$$H^0 ( \sQ_{L^\tJ} ( \beta ), \cO_{\sQ_{L^\tJ} ( \beta )} ( -w_0 \la ) )^* \hookrightarrow H^0 ( \sQ_{G} ( \beta ), \cO_{\sQ_{G} ( \beta )} ( -w_0 \la ) )^* \hookrightarrow \bW _G ( \la ).$$
Combined with (\ref{ind-W}), we conclude the result.
\end{proof}

\begin{prop}\label{Aqt}
Let $i \in \tI$. Find $i' \in \tI$ such that $\al_{i'} = w_0 \al_i$. The $A^{\pm \varpi_{i}}$-action on $qK_{\bG} ( \sB )$ is the same as the tensor product of $\cO _{\sB} ( \pm \varpi_{i} )$ on $K_{\bG} ( \sB )^{\wedge}$ modulo $Q_{i'}$.
\end{prop}

\begin{proof}
Let $\tJ' := \tI \setminus \{i'\}$. By our definition of $A^{\pm \varpi_i}$, it suffices to see
\begin{equation}
\left< A^{\pm \varpi_{i}} a, b \right>_{G}^{\GW} \equiv \left< \cO _{\sB} ( \pm \varpi_{i} ) \otimes a, b\right>_{G}^{\GW} \mod Q_{i'}\label{Ai}
\end{equation}
for every $a, b \in K_{\bG} ( \sB )$. Since $K_{\bG} ( \sB )$ is generated by $A^{\la}$ for $- \la \in \La_+$ and $Q^{\beta}$ ($\beta \in Q^{\vee}_+$) as $\C_q \bX^* ( G )$-algebra, we can take $a = A^{\mu}$ and $b = [\cO_{\sB}]$. Since $\sQ_G ( \beta )$ has rational singularities for every $\beta \in Q^{\vee}_+$ (Theorem \ref{rat}), we have
$$\left< A^{\pm \varpi_{i} + \la} [\cO_{\sB}], [\cO_{\sB}] \right>_{G}^{\GW} = \sum_{\beta \in Q^{\vee}_+} Q^{\beta} \chi ( \sQ_G ( \beta ), \cO_{\sQ_G ( \beta )} ( \pm \varpi_{i} + \la ) ) ) \hskip 5mm \la \in \bX^*.$$
In case $\left< \beta, \varpi_{i'} \right> = 0$, the structure map $\sQ_{L^{\tJ'}} ( \beta )\to\mathrm{pt}$ and Lemma \ref{Q-ind} yield a projection map $\eta : \sQ_G ( \beta ) \rightarrow G / P^{\tJ'} = \sB_{\tJ'}$, that is $G$-equivariant. This implies
\begin{equation}
\chi ( \sQ_G ( \beta ), \cO_{\sQ_G ( \beta )} ( \la ) ) ) = D_{w_0} ( e^{- \left< \al_{i}^{\vee}, \la \right> \varpi_{i'}} \chi ( \sQ_L ( \beta ), \cO_{\sQ_L ( \beta )} ( \la - \left< \al_{i}^{\vee}, \la \right> \varpi_{i} ) ) )\label{infEuler}
\end{equation}
for each $\la \in \bX^*$. The twist by $e^{- \varpi_{i'}}$ in the RHS of (\ref{infEuler}) is just a $\cO (1)$-line bundle twist of $\sB_{\tJ'}$ pulled back by $\eta$. Thus, it arises from the line bundle twist of $\cO _{\sB} ( \varpi_i )$ through $\mathsf{ev}_1$. Therefore, we conclude (\ref{Ai}) as required.
\end{proof}

\subsection{Proof of Theorem \ref{qK-surj}}\label{sec:qKs}
This subsection is entirely devoted to the proof of Theorem \ref{qK-surj}. We set $\tJ^\# := \{i \in \tI \mid \al_i = -w_0 \al_j, \, j \in \tI \setminus \tJ \}$ and $\tJ' := \{i \in \tI \mid \al_i = -w_0 \al_j, \, j \in \tJ \}$.

By Theorem \ref{reconst}, we know that $qK_{\bL ^ \tJ } ( \sB^{\tJ} )$ is generated from $[\cO_{\sB^{\tJ}}]$ by $A^{\pm w_0 \varpi_i}$ ($i \in \tJ$), $Q_i$ ($i \in \tJ'$), and $\bX^*_0 ( \tJ )$ as an algebra. Suppose that
$$f ( e^{\mu}, x_i, Q ) = \sum _{\vec{m} \in \Z^r, \mu \in \bX^*_0 ( \tJ  ), \gamma \in Q^{\vee}_{\tJ, +}} f_{\vec{m}, \mu, \beta} e^{\mu} x^{\vec{m}} Q^{\gamma} \in \C_q \bX^*_0 ( \tJ ) [x_1^{\pm 1},\ldots,x_r^{\pm 1}][\![Q^{\vee}_{\tJ,+}]\!],$$
where $x^{\vec{m}} := x^{m_1}_1 \cdots x^{m_r}_r$ for $\vec{m} = ( m_1,\ldots,m_r )$, satisfies
$$f ( e^{\mu}, A^{\varpi_i}, Q ) [\cO_{\sB^{\tJ}}] = 0 \in qK_{\bL ^\tJ} ( \sB^{\tJ} ),$$
where $A^{\pm w_0\varpi_i}$ is interpreted as $e^{\mp \varpi_i}$ for $i \not\in \tJ$. The line bundle $\C_{\mu} \otimes \cO_{\sQ_{L^\tJ} ( \beta )} ( -w_0 \la )$ for $\beta \in Q^{\vee}_{\tJ',+}$, $\mu \in \bX^*_0 ( \tJ )$, and $\la \in \La^\tJ$ inflates to $\cO_{\sQ_{G} ( \beta )} ( - w_0 ( \la + \mu ) )$ by Lemma \ref{Q-ind} and (\ref{Pemb}). Let
$$\tilde{f} ( e^{\mu}, A^{\varpi_i}, Q ) = \sum _{\vec{m} \in \Z^r, \nu \in \bX^* ( G ), \beta \in Q^{\vee}_{\tJ, +}} \tilde{f}_{\vec{m}, \nu, \beta} e^{\nu} x^{\vec{m}} Q^{\beta} \in \C_q \bX^* ( G ) [x_1^{\pm 1},\ldots,x_r^{\pm 1}][\![Q^{\vee}_{+}]\!]$$
be the polynomial obtained from $f$ by replacing $e^{-\varpi_{i}}$ with $x_{i'}$ (for each $i \in \tI \setminus \tJ$ and $i' \in \tI$ such that $\varpi_i = -w_0 \varpi_{i'}$). For each $\la \in \La$, we have
\begin{align*}
\left< A^\la \tilde{f} ( e^{\mu}, A^{\varpi_i}, Q ) [\cO_{\sB}], [\cO_{\sB}] \right>^{\GW}_{G} & \\
 = \sum_{\beta \in Q^{\vee}_+} \sum _{\vec{m}, \nu, \gamma} & \tilde{f}_{\vec{m}, \nu, \gamma} Q^{\beta + \gamma} e^{\nu} \chi ( \sX _G ( \beta ), \cO_{\sX_G} ( \la + \sum_{i \in \tI} m_i \varpi_i ) )\\
= \sum_{\beta \in Q^{\vee}_+} \sum _{\vec{m}, \nu, \gamma} & \tilde{f}_{\vec{m}, \nu, \gamma} Q^{\beta + \gamma} e^{\nu} \chi ( \sQ _G ( \beta ), \cO_{\sQ_G} ( \la + \sum_{i \in \tI} m_i \varpi_i ) )\\
\equiv \sum_{\beta \in Q^{\vee}_{\tJ',+}} \sum _{\vec{m},\nu,\gamma} e^{\nu} D_{w_0} (\tilde{f}_{\vec{m},\nu,\gamma} Q^{\beta + \gamma} & \, \chi ( \sQ _{L ^\tJ} ( \beta ), \cO_{\sQ_{L ^\tJ}} ( \la + \sum_{i \in \tI} m_i \varpi_i ) ) )\\
& \hskip 35mm \mod ( Q_i \mid i \in \tJ^\#),
\end{align*}
where the first equality is the the definition, the second equality follows from Theorem \ref{rat}, and the third equality follows from Lemma \ref{Q-ind} and the fact that $\cO_{\sQ_{L ^\tJ}} ( \la )$ is the restriction of $\cO_{\sQ_{G}} ( \la )$. Similarly, we have
\begin{align*}
0 = \left< A^\la f ( e^{\mu}, A^{\varpi_i}, Q ) [\cO_{\sB^{\tJ}}], [\cO_{\sB^{\tJ}}] \right>^{\GW}_{L ^ \tJ } & \\
= \sum_{\beta \in Q^{\vee}_+} \sum _{\vec{m}, \mu, \gamma} f_{\vec{m}, \mu, \gamma} \, & Q^{\beta + \gamma} e^{\mu} \chi ( \sQ_{L ^ \tJ } ( \beta ), \cO_{\sQ_{L ^ \tJ }} ( \la + \sum_{i \in \tI} m_i \varpi_i ) )
\end{align*}
for $\la \in \Lambda$. By examining the relation between $f$ and $\widetilde{f}$, we conclude
$$\left< A^\la \tilde{f} ( e^{\mu}, A^{\varpi_i}, Q ) [\cO_{\sB}], [\cO_{\sB}] \right>^{\GW}_{G} \equiv 0 \mod ( Q_i \mid i \in \tJ^\#)$$
for $\la \in \La$. In view of Theorem \ref{reconst}, this is equivalent to
$$\tilde{f} ( e^{\mu}, A^{\varpi_i}, Q ) [\cO_{\sB}] \equiv 0 \mod ( Q_i \mid i \in \tJ^\#).$$
This yields a map $qK_{\bG} ( \sB ) \rightarrow qK_{\bL ^ \tJ } ( \sB ^{\tJ} )$ that intertwines $A^{\la}$ ($\la \in \La$), $Q_i$ ($i \in \tI$), and $\C_q \bX^* (G)$-actions. The $Q_i \equiv 0$ ($i \in \tI$) specialization of this map is the restriction map, that is an isomorphism (as a consequence of the bijection between equivariant line bundles through the restriction; cf. Corollary \ref{K-rest}). Since the $\C Q^{\vee}_{\tJ',+}$-actions are free on the both of $qK_{\bG} ( \sB ) / ( Q_i \mid i \in \tJ^\#)$ and $qK_{\bL ^\tJ} ( \sB ^{\tJ} )$, we conclude that
$$qK_{\bG} ( \sB ) / ( Q_i \mid i \in \tJ^\# ) \stackrel{\cong}{\longrightarrow} qK_{\bL ^ \tJ} ( \sB ^{\tJ} )$$
as required.

\section{Finkelberg-Tsymbaliuk's conjecture}

We work in the settings of \S \ref{sec:setup}. The goal of this section is to prove the following main theorem of this paper, originally conjectured by Finkelberg-Tsymbaliuk \cite{FT19}:

\begin{thm}\label{main}
Let $G$ be a connected reductive algebraic group over $\C$ such that $[G,G]$ is simply connected and $G \cong [G, G] \times H'$ for a subtorus $H' \subset H$. Let $L$ be a reductive subgroup that contains $H$. The embedding of Corollary \ref{absDarboux} induces algebra embeddings
$$K_{\bG} ( \Gr_G ) \hookrightarrow K_{\bL} ( \Gr_L ) \hookrightarrow K_{\bH} ( \Gr_H ).$$
\end{thm}

Theorem \ref{main} is proved in \S \ref{subsec:main}. From Theorem \ref{main}, we conclude the following enhancement:

\begin{cor}\label{mcor}
Let $G$ be a connected reductive algebraic group over $\C$ such that $[G,G]$ is simply connected and $[G, G] \times H'$ for a subtorus $H' \subset H$. Let $L$ be a connected reductive subgroup of $G$ that contains $H$. Let $Z \subset H \cap Z ( G )$ be a finite subgroup. Theorem \ref{main} induces embeddings
$$K_{\bG} ( \Gr_{G / Z} ) \hookrightarrow K_{\bL} ( \Gr_{L / Z} ) \hookrightarrow K_{\bH} ( \Gr_{H/Z} )$$
of algebras.
\end{cor}

\begin{proof}
We set $G' := G/Z, L' := L/Z$. Note that the quotient $H \to H / Z$ induces an injective map
$$\bX_* \cong \Gr_{H} \longrightarrow \Gr_{H/Z}$$
that identifies $\bX_*$ with a subset of the group of cocharacters $\bX_*'$ of $H/Z$ via the quotient map. This gives rise to an isomorphism
$$K_{\bH} ( \Gr_{H/Z} ) \cong \bigoplus_{\chi \in \mathsf{Irr}\, Z} K_{\bH} ( \Gr_H )$$
of algebras. In particular, the connected components of $\Gr_{H/Z}$ is the union of the contributions
$$\Gr_{H / Z} = \bigsqcup_{\chi \in \mathsf{Irr}\, Z} \Gr_{H/Z}^{\chi}.$$
The same is true for $\Gr_{G'}$ and $\Gr_{L'}$, that we denote by
$$\Gr_{G'} = \bigsqcup_{\chi \in\mathsf{Irr}\, Z} \Gr_{G'}^{\chi} \hskip 5mm \text{and} \hskip 5mm \Gr_{L'}^{\chi} = \bigsqcup_{\chi \in \mathsf{Irr}\, Z} \Gr_{L'}^{\chi}.$$
Note that the content of Theorem \ref{main} under this setup is the algebra embeddings:
\begin{equation}
K_{\bG} ( \Gr_{G'}^1 ) \hookrightarrow K_{\bL} ( \Gr_{L'}^1 ) \hookrightarrow K_{\bH} ( \Gr_{H/Z}^1 ),\label{trivDC}
\end{equation}
where $1 \in \mathsf{Irr} \, Z$ is the trivial representation.

The action of $\bX_*' / \bX_*$ induces outer automorphisms of the affine Dynkin diagram of $G$. This twists the embedding $K_{\bG} ( \Gr_{G'}^{\chi} ) \subset K_{\bH} ( \Gr_{G'}^{\chi} )$ into $K_{\bG} ( \Gr_{G'}^{1} ) \subset K_{\bH} ( \Gr_{G'}^{1} )$ by the Dynkin diagram automorphisms. These outer automorphisms induce automorphisms of $\sH_q$, and hence gives rise to an algebra structure of $K_{\bG} ( \Gr_{G'} )$ induced from $K_{\bH} ( \Gr_{G'} )$. If we employ these twists of $R ( \bH )$ also to the coefficients of $K_{\bH} ( \Gr_{H/Z}^{\chi} )$, we obtain embeddings
\begin{equation}
K_{\bG} ( \Gr_{G'}^{\chi} ) \longrightarrow K_{\bH} ( \Gr_{H/Z}^{\chi} ) \hskip 5mm \chi \in \mathrm{Irr} \, Z.
\label{twistDC}
\end{equation}
Such twists, altogether along $\mathsf{Irr} \, Z$, give rise to a twist of the algebra structure of $K_{\bH} ( \Gr_{H/Z} )$ (that prolongs $K_{\bH} ( \Gr_{H/Z}^1 ) \cong K_{\bH} ( \Gr_{H} )$). With these twisted algebra structures, we obtain a morphism
$$K_{\bG} ( \Gr_{G'} ) \longrightarrow K_{\bH} ( \Gr_{H/Z} )$$
of algebras that prolongs (\ref{trivDC}) and (\ref{twistDC}).

It remains to find that such a twisting can be taken to be compatible with the analogously defined embedding $K_{\bL} ( \Gr_{L'} ) \subset K_{\bH} ( \Gr_{H/Z} )$. To see this, it is enough to mind that the twisting by $\chi \in \mathsf{Irr}\,Z$ gives a twisting of $G' [\![z]\!] \subset G' (\!(z)\!)$ by a lift of $\chi$ in $\bX_* '$ (up to internal automorphism), and it naturally induce a twisting of $L'[\![z]\!] \subset G' (\!(z)\!)$.
\end{proof}

\begin{ex}\label{SL2}
We assume that $G = \mathop{SL} ( 2 )$ and $L = H$ is its maximal torus. We have $Q^{\vee} = \bX_* = \Z \al^{\vee}$, where $\al^{\vee}$ is the positive simple coroot of $G$. Let $\varpi$ be the fundamental weight. We have
$$R ( G ) = \C [e^{\pm \varpi}]^{\mathfrak S_2} \subset \C [e^{\pm \varpi}] = R ( H ).$$
Theorem \ref{main} yields an algebra map
$$R ( G ) \equiv R ( G ) [\mathcal O_{\Gr_G ( 0 )}] \hookrightarrow K_G ( \Gr_G ) \longrightarrow K_H ( \Gr_H ) = \bigoplus_{\gamma \in Q^{\vee}} R ( H ) \mt _\gamma,$$
where $\mt_\gamma$ represents the class of the structure sheaf of $\Gr_H ( \gamma )$, that is a point. In view of Proposition \ref{resGH}, we find that
$$[\mathcal O_{\Gr_G ( 0 )}] \mapsto \mt_{0}, \hskip 5mm ( e^{\varpi} + e^{- \varpi} ) [\mathcal O_{\Gr_G ( 0 )}] \mapsto e^{\varpi} (\mt_0 - \mt_{\al^{\vee}}) + e^{- \varpi} \mt_0.$$
\end{ex}

We remark that Example \ref{SL2} is obtained from the $n = 2$ case of Example \ref{exToda} by applying Theorem \ref{main} (and Theorem \ref{ncPet}).

\subsection{Classes $E (\beta, \la )$ and $\cO^{\star} ( \la )$}
We find $\tJ \subset \tI$ such that $L$ in Theorem \ref{main} is written as $L^\tJ$. For $\beta \in \bX_*^{\le} ( \tJ )$, we set $\tJ ( \beta ) = \{j \in \tJ \mid \left< \al_j^{\vee}, \beta \right> = 0 \} \subset \tJ$. We set $w ( \tJ, \beta ) := w_0 ^{\tJ} w_0 ^{\tJ ( \beta )} w_0 ^{\tJ}$ and $\tJ ( \beta )^{\#} := \{j \in\tJ \mid \exists j' \in \tJ ( \beta ) \text{ s.t. } \varpi_j = - w_0^{\tJ} \varpi_{j'} \} $ (i.e. $w ( \tJ, \beta ) = w_0 ^{\tJ ( \beta )^\#}$). We set $\La_{+}^{\tJ} ( \beta ) := \La^{\tJ \setminus \tJ ( \beta )} + \La_+^{\tJ ( \beta )}$. For each $\la \in \La_{+}^{\tJ} ( \beta )$, we define
$$E^{\tJ} [\beta; \la] := D_{w_0^{\tJ}} ( e^{w_0^{\tJ} \la} [\cO_{\Gr_{L} ( u^\tJ_{\beta} )}] ) \in K_{\bL} ( \Gr_{\bL}),$$
where $u^{\tJ}_{\beta} \in W^{\tJ} t_{\beta} W^{\tJ}$ is the minimal length element inside the double coset.

\begin{lem}\label{E-basis}
The $\sH_q ( \tJ )$-module $K_{\bL} ( \Gr_L )$ admits a direct sum decomposition whose associated graded pieces are parametrized by $\bX_* ^{\le} ( \tJ )$. The associated graded piece corresponding to $\beta$ is isomorphic to $K_{\bL} ( \sB^\tJ_{\tJ ( \beta )^\#})$ and the correspondence is given by
$$E^{\tJ} [\beta; \la] \mapsto D_{w_0^{\tJ}} ( e^{w_0^{\tJ} \la} D_{w (\tJ, \beta )} [\cO_{\sB^{\tJ} ( w_0^{\tJ} )}] ) \hskip 5mm \la \in \La_{+}^{\tJ} ( \beta ).$$
In particular, the set $\{ E^{\tJ} [\beta; \la] \}_{\beta \in \bX_* ^{\le} ( \tJ ), \la \in \La^{\tJ}_+ (\beta)}$ forms a $\C_q \bX^*_0 ( \tJ )$-basis of $K_{\bL} ( \Gr_{\bL})$.
\end{lem}

\begin{proof}
By definition, we have a $\C [\bH]$-basis of $K_{\bH} ( \Gr_{\bL})$ offered by $[\cO_{\Gr_{\bL} ( w t_{\beta} )}]$ for $\beta \in \bX^{\le}_* ( \tJ )$ and $w \in W^{\tJ} / W^{\tJ ( \beta )}$. We have $K_{\bL} ( \Gr_{\bL}) = D _{w_0^{\tJ}} ( K_{\bH} ( \Gr_{\bL}) )$. By the Leibniz rule of $D_i$ for each $i \in \tI$ (Lemma \ref{Lei}), we conclude that the space of $D _{w_0^{\tJ}}$-invariants in $K_{\bH} ( \Gr_{\bL})$ is the direct sum of the $D _{w_0^{\tJ}}$-invariants in
\begin{equation}
\bigoplus_{w \in W^{\tJ} / W^{\tJ ( \beta )}} \C [\bH] [\cO_{\Gr_{\bL} ( w t_{\beta} )}]\label{orbB}
\end{equation}
for all $\beta \in \bX_* ^{\le} ( \tJ )$. The space (\ref{orbB}) is stable under the action of $D_j$ ($j \in \tJ$) again by the Leibniz rule. In addition, it is generated from $[\cO_{\Gr_{L} ( u^\tJ_{\beta} )}]$, that is $D_{w (\tJ, \beta )}$-invariant as $s_i \beta = \beta$ for $i \in \tJ ( \beta )$. By Corollary \ref{Rsph} (and Theorem \ref{finK}), we deduce that (\ref{orbB}) is isomorphic to $K_{\bH} ( \sB^\tJ_{\tJ ( \beta )^{\#}} )$ as $\sH_q ( \tJ )$-module via the assignment
$$[\cO_{\Gr_{L} ( u^\tJ_{\beta} )}] \mapsto D_{w ( \tJ, \beta )} ( [\cO_{\sB^\tJ ( w_0^\tJ )}] ).$$
This yields the desired correspondence between elements. Note that we have some $u \in W^\tJ$ such that $w_0^{\tJ} = u w ( \tJ, \beta )$ and $\ell ( w_0^{\tJ} ) = \ell ( u ) + \ell ( w ( \tJ, \beta ) )$. It follows that
$$D_{w_0^{\tJ}} ( e^{w_0^{\tJ} \la} D_{w ( \tJ, \beta )} [\cO_{\sB^{\tJ} ( w_0^{\tJ} )}] ) = D_{u} \left( D_{w (\tJ, \beta )} ( e^{w_0^{\tJ}\la} D_{w ( \tJ, \beta )} [\cO_{\sB^{\tJ} ( w_0^{\tJ} )}] ) \right),$$
represents a $\bL$-equivariant vector bundle whose fiber is a $L^{\tJ ( \beta )^\#}$-module with its character $D_{w ( \tJ, \beta )} ( e^{w_0^{\tJ} \la} )$. The latter is $\ch \, V^{\tJ ( \beta )^\#} ( w ( \tJ, \beta ) w_0^{\tJ} \la )$ by the Weyl character formula. We have
$$K_{\bL} ( \sB^\tJ_{\tJ (\beta)^\#}) \cong R ( \mathbf P ^{\tJ (\beta)^\#} ) = R ( \bL ^{\tJ (\beta)^\#} ),$$
and the set of characters $\ch \, V^{\tJ ( \beta)^\#} ( w ( \tJ, \beta ) w_0^{\tJ} \la )$ for $\la \in \La _{+}^{\tJ} ( \beta )$ is a $\C_q \bX^*_0 ( \tJ )$-basis of $R ( \bL ^{\tJ (\beta)^\#} )$. Therefore, we conclude that $\{ E^{\tJ} [\beta; \la] \}_{\la \in \La^{\tJ}_+ (\beta)}$ is the $\C_q \bX^*_0 ( \tJ )$-basis of the $D_{w_0^{\tJ}}$-invariant part of (\ref{orbB}). Since $K_{\bL} ( \Gr_{\bL})$ is the direct sum of $D_{w_0^{\tJ}}$-invariant parts of (\ref{orbB}), we conclude the result.
\end{proof}

We set $E^{\tJ}_{\mathrm{st}} [\gamma; \la] := E^{\tJ} [\gamma + \beta ; \la] \odot_q \mt_{- \beta}$ for $\la \in \La^\tJ,  \gamma \in \bX_*, \beta, \beta + \gamma \in \bX_*^- ( \tJ )$.

\begin{cor}\label{stabE}
The element $E^{\tJ}_{\mathrm{st}} [\gamma; \la]$ does not depend on the choice $($of $\beta)$.
\end{cor}

\begin{proof}
The assertion follows from the fact that the right action of $\mt_{\beta}$ commutes with the left action of $D_i$ $(i \in \tJ)$.
\end{proof}

By construction, we have $L \cong H'' \times [L, L]$ for a connected subtorus $H'' \subset H$. In particular, we have
$$L \cong H'' \times \prod_{k = 1}^n L _k$$
where each $L_k$ is a simply connected simple algebraic group. Let $Q^{\vee}_k \subset Q^{\vee}$ be the span of simple coroots corresponding to (co-)roots in $L_k$. We have
\begin{equation}
K_{\bL} ( \Gr_{\bL}) \cong K_{\bH''} ( \Gr_{H''} ) \otimes_{\C_q} \bigotimes_{k=1}^n K_{\bL_k} ( \Gr_{L_k}),\label{defKGr}
\end{equation}
where the big tensor product is also taken over $\C_q$. On $K_{\bL} ( \Gr_{\bL})$, we have the translation elements $\mt_{\beta}$ for each $\beta \in \bX_*$ obtained as the product of $\mt_{\gamma}$'s that act on one of the tensor factors. This makes (\ref{defKGr}) into the isomorphism between their localized versions.

Using this, we consider the maps $\Psi_{\tJ}$ and $\Phi_{\tJ}'$ obtained from these of Theorem \ref{K-id} and Theorem \ref{qK-int} by employing the following spaces:
$$K_{\bL} ( \bQ_{\tJ}^{\ra} ) := \bigotimes_{k = 1}^n K_{\bL_k} ( \bQ_{L_k}^{\ra} ) \otimes K_{\bH''} ( \Gr_{H''} ) \hskip 5mm \text{and} \hskip 5mm qK_{\bL} ( \sB^{\tJ} )_\lo \otimes K_{\bH''} ( \Gr_{H''} ),$$
where all the tensor products are taken over $\C_q$, the $\Phi_{\tJ}$ is $K_{\bH''} ( \Gr_{H''} )$-linear, and the map $\Psi_{\tJ}'$ is also $K_{\bH''} ( \Gr_{H''} )$-linear, though the Novikov variables and line bundles (including the Heisenberg generators of $K_{\bH''} ( \Gr_{H''} )$) are twisted by $-w_0$ from its naive definition. Note that the multiplication by $\mt_{\beta}$ ($\beta \in \bX_*$) corresponds to $Q^{-w_0 \beta}$ only if $\beta \in Q^{\vee}_{\tJ}$, and the multiplication by $Q^{\beta}$ for $\bX_*$ is extended formally.

\begin{lem}\label{transO}
For $\beta \in \bX_*$ and $\la \in \La^{\tJ}$, we have
$$E^{\tJ}_{\mathrm{st}} [\beta ; \la] = \Phi^{-1}_{\tJ} \circ \Psi'_{\tJ} ( [\cO_{\sB ^{\tJ}} ( -w_0 \la )] Q^{-w_0 \beta} ).$$
In particular, the set $\{ E^{\tJ}_{\mathrm{st}} [\beta; \la] \}_{\beta \in \bX_*, \la \in \La^{\tJ}}$ is a $\C_q \bX^*_0 ( \tJ )$-basis of $K_{\bL ^\tJ} ( \Gr_{\bL ^\tJ})_\lo$.
\end{lem}

\begin{proof}
We have $[\cO_{\sB^{\tJ}} ( \la )] = D _{w_0^{\tJ}} ( e^{w_0^{\tJ}\la} [\cO_{\sB^{\tJ} ( w_0^{\tJ} )}] ) \in K_{\bH} ( \sB^\tJ )$. In view of the correspondence between Schubert classes under the maps $\Psi$ \cite[Theorem 4.1 and its proof]{Kat18c} and $\Phi$ \cite[Proposition 2.13 and Remark 2.14]{Kat18c}, we deduce the first assertion. Taking into account of the first assertion and Theorem \ref{K-id}, the second assertion follows from Theorem \ref{dcf} and Theorem \ref{X-weight}.
\end{proof}

\begin{lem}
The embedding of Proposition \ref{resGH} induces algebra embeddings
$$K_{\bG} ( \Gr_G )_\lo \hookrightarrow K_{\bL} ( \Gr_L )_\lo \hookrightarrow K_{\bH} ( \Gr_H ).$$
\end{lem}

\begin{proof}
In view of Corollary \ref{genKG} and Proposition \ref{resGH}, we find that $K_{\bG} ( \Gr_G )_\lo$ and $K_{\bL} ( \Gr_L )_\lo$ are obtained by replacing the generator $e^{\varpi_i}$ $(i \in\tI)$ in $K_{\bH} ( \Gr_H )$ to $\xi_i$ for $i \in \tJ$ ($e^{- \varpi_i}$ and $\phi_i$ are the same for every $i \in \tI$). The commutation relation in Proposition \ref{resGH} implies $K_{\bG} ( \Gr_G )_\lo \subset K_{\bL} ( \Gr_L )_\lo$ inside $K_{\bH} ( \Gr_H )$.
\end{proof}

For $\la \in \La$, we write $\la = \sum_{j \in \tI} m_j \varpi_j$ for some $m_j \in \Z$. For each $\beta \in \bX_*$, we define
$$[\cO_{\beta} ^\star ( \la ) ] := \left( \prod_{j \in \tI, m_j < 0} \phi_i ^{- m_j} \right) \left( \prod_{j \in \tI, m_j > 0} \xi_i ^{m_j} \right) ( \mt_\beta ) \in K_{\bG} ( \Gr_{G} )_\lo.$$
Similarly, for each $\la \in \La$, we write $\la = \mu + \sum_{j \in \tJ} m_j \varpi_j$ for some $\mu \in \La^{\tI \setminus \tJ}$ and $m_j \in \Z$, and we define
$$[\cO_{\tJ, \beta} ^\star( \la ) ] := e^{\mu} \left( \prod_{j \in \tJ, m_j < 0} \phi_i ^{- m_j} \right) \left( \prod_{j \in \tJ, m_j > 0} \xi_i ^{m_j} \right) ( \mt_\beta ) \in K_{\bL} ( \Gr_{L} )_\lo.$$

\begin{lem}\label{O-basis}
For $\la \in \La^{\tJ}$, we have
$$[\cO_{\tJ, 0}^\star ( \la )] = E ^{\tJ}_{\mathrm{st}} [ 0; \la ] \mod (\mt_{\al_j^{\vee}} \mid j \in \tJ).$$
\end{lem}

\begin{proof}
In view of Theorem \ref{sib} and Theorem \ref{K-id}, the assertion follows from Theorem \ref{HQc} 4) and the definitions of $\phi_i$'s and $\xi_i$'s.
\end{proof}

By the comparison of Lemma \ref{E-basis} and Lemma \ref{O-basis}, we have a transition matrix (that is a finite sum in view of Corollary \ref{genKG})
$$E^{\tJ} [\beta; \la] = \sum_{\gamma \in \bX_*, \mu \in \La^{\tJ}} a_{\beta,\la} ^{\gamma, \mu} ( \tJ ) [\cO_{\tJ, \gamma}^\star ( \mu )]$$
for $a_{\beta,\la} ^{\gamma, \mu} ( \tJ ) \in \C_q \bX^*_0 ( \tJ )$. Moreover, we have:

\begin{lem}
We have $a_{\beta,\la} ^{\beta, \la} ( \tJ ) = 1$, and
$$a_{\beta,\la} ^{\gamma, \mu} ( \tJ ) = 0 \hskip 5mm \text{for every} \hskip 5mm \gamma \not\in \beta + Q^{\vee}_{\tJ,+}.$$
\end{lem}

\begin{proof}
The assertion follows by Lemma \ref{O-basis} and the fact that the effect of line bundle twists of $\bQ_{L_k}$ raises the translation parts by $Q^{\vee}_{\tJ,+}$.
\end{proof}

\begin{prop}\label{a-comp}
For each $\la \in \La^\tJ$ and $\beta \in \bX_* ^{-}$, we have
$$a_{\beta,\la} ^{\gamma, \mu} ( \tJ ) = a_{\beta,\la} ^{\gamma, \mu} \hskip 5mm \gamma \in \beta + Q^{\vee}_{\tJ,+}.$$
\end{prop}

\begin{proof}
By assumption, we have $E  [ \beta; \la ] = E _{\mathrm{st}} [ \beta; \la ]$ and $E ^{\tJ} [ \beta; \la ] = E ^{\tJ}_{\mathrm{st}} [ \beta; \la ]$. Thanks to Theorem \ref{K-id} and Theorem \ref{qK-int}, we transplant the problem to the quantum $K$-groups via $( \Psi_{\tJ}')^{-1} \circ \Phi_{\tJ}$. In view of Corollary \ref{K-rest}, the assertion follows by Theorem \ref{qK-surj} and Lemma \ref{transO}.
\end{proof}

\begin{prop}\label{aa-comp}
For each $\beta \in \bX_* ^{\le}$ and $\la \in \La_+ ( \beta )$, we have
$$a_{\beta,\la} ^{\gamma, \mu} = \sum_{\la'} c _{\la'} a_{\beta,\la'} ^{\gamma, \mu} ( \tJ ) \hskip 5mm \gamma \in \beta + Q^{\vee}_{\tJ,+},$$
where $\la' \in \La^{\tJ}_+ ( \beta )$ and $c_{\la'} \in \C_q \bX^*_0 ( \tJ )$.
\end{prop}

\begin{proof}
We borrow the setting in the proof of Lemma \ref{E-basis}. The element $E [\beta; \la]$ corresponds a $G$-equivariant vector bundle over $\sB_{\tI ( \beta )^\#}$ inflated from a $L^{\tI ( \beta )}$-module $V^{\tI ( \beta )} ( \la )$, while the element $E^{\tJ} [\beta; \la']$ corresponding to a $L^{\tJ}$-equivariant vector bundle over $\sB^{\tJ}_{\tJ ( \beta )^{\#}}$ inflated from a $L^{\tJ ( \beta )}$-module $V^{\tJ ( \beta )} ( \la' )$. These are parametrized by $\La_+ ( \beta )$ and $\La^{\tJ}_+ ( \beta )$, respectively. In particular, we have
\begin{equation}
V^{\tI ( \beta )} ( \la ) \cong \bigoplus_{\la' \in \La^{\tJ}_+ ( \beta )} V^{\tJ ( \beta )} ( \la' )^{\oplus c_{\la'}},\label{Vrest}
\end{equation}
where $c_{\la'} \in \C_q \bX^*_0 ( \tJ ) \subset \C_q \bX^*$ is understood to be the multiplicity space that carries the information of character twists.

Consider the expansions
$$E^{\tJ} [\beta; \la] = \sum_{\mu} d_{\mu}^\la E^{\tJ}_{\mathrm{st}} [\beta; \mu] \hskip 2mm (\la \in \La^{\tJ (\beta)}_+)\hskip 2mm \text{and} \hskip 2mm E [\beta; \la]  = \sum_{\mu} e_{\mu}^\la E_{\mathrm{st}} [\beta; \mu] \hskip 2mm (\la \in \La^{\tI (\beta)}_+)$$
with $d_{\mu}^\la \in \C_q \bX^*_0 ( \tJ ), e_{\mu}^\la \in \C_q \bX^* ( G )$. These correspond to the expansions of the pullbacks of the class of vector bundles on $\sB^\tJ_{\tJ ( \beta )^\#}$ and $\sB_{\tI ( \beta )^\#}$ to $\sB^\tJ$ and $\sB$ in terms of line bundles by Corollary \ref{Rsph}, respectively. It respects the decomposition through the comparison given by Corollary \ref{K-rest}, that sends $E _{\mathrm{st}} [\beta; \la]$ ($\la \in \La$) to $e^{\la - \la'} E^{\tJ} _{\mathrm{st}} [\beta; \la']$ for $\la' \in \La^{\tJ}$ such that $\la - \la' \in \La^{\tI \setminus \tJ}$.

It follows that
$$d_{\mu}^\la = \sum_{\la'} c_{\la'} e_{\mu}^{\la'}.$$
Now the assertion follows by transplanting the problem to the quantum $K$-groups via $( \Psi_{\tJ}')^{-1} \circ \Phi_{\tJ}$ thanks to Proposition \ref{qK-surj}.
\end{proof}

\subsection{Proof of Theorem \ref{main}}\label{subsec:main}
This subsection is totally devoted to the proof of Theorem \ref{main}.  We consider elements of $K_{\bG} ( \Gr_{G} ) $ and $K_{\bL} ( \Gr_{L} )$ as elements of $K_{\bH} ( \Gr_{H} )$ via Corollary \ref{absDarboux}. Since we have $\phi_i, \xi_i, \mt_{\pm \al_i^{\vee}} \in K_{\bL} ( \Gr_{L} )$ for $i \not\in \tJ$, we have
\begin{equation}
K_{\bG} ( \Gr_G ) \subset K_{\bL} ( \Gr_{L} )\label{geninc}
\end{equation}
if and only if
\begin{equation}
K_{\bG} ( \Gr_G ) [\phi_i, \xi_i, \mt_{\pm \al_i^{\vee}} \mid i \not\in \tJ] \subset K_{\bL} ( \Gr_{L} ),\label{locinc}
\end{equation}
where the LHS exist as a subalgebra of $K_{\bH} ( \Gr_H )$. We consider the completions of the both sides of (\ref{locinc}) using the variables $\{ \mt_{\beta} \}_{\beta \in \bX_*}$ with respect to the direction $\left< \beta, \varpi_i \right> \to \infty$ for $i \not\in \tJ$. We denote the completion of the LHS of (\ref{locinc}) by $\mathbf K_G^{\wedge}$ and the completion of the RHS of (\ref{locinc}) by $\mathbf K_L ^{\wedge}$. We have $( \sum_{k = 0}^{\infty}\mt_{k\al_i^{\vee}}) \xi _i \in \mathbf K_G^{\wedge}$ for $i \not\in \tJ$, that is an inverse of $\phi_i$. We have (\ref{geninc}) if and only if $\mathbf K_G^{\wedge} \subset \mathbf K_L ^{\wedge}$.

For a collection $\vec{m} := \{ m_i \}_{i \in (\tI \setminus \tJ)} \in \Z^{(\tI \setminus \tJ)}$, we set $\La ( \vec{m} ) : = \{ \la \in \La \mid \left< \al_i^{\vee}, \la \right> = m _i , i \in (\tI \setminus \tJ) \}$. Assume that
$$\sum_{\la \in \La, \beta \in \gamma + Q_{+}^{\vee}} c_{\la,\beta} [\cO_\beta^\star ( \la ) ] \in K_{\bG} ( \Gr_G ) \hskip 5mm c_{\la, \beta} \in \C_q \bX^* ( G ).$$
By taking the conjugations by $\mt_{\al_i^{\vee}}$ for each $i \in ( \tI \setminus \tJ )$ and separate out the eigenvectors, we conclude that
$$\sum_{\la \in \La ( \vec{m} ), \beta \in \gamma + Q_{+}^{\vee}} c_{\la,\beta} [\cO_\beta^\star ( \la ) ] \in K_{\bG} ( \Gr_G ) [\phi_i, \xi_i, \mt_{\pm \al_i^{\vee}} \mid i \not\in \tJ].$$
Inside $\mathbf K_G^{\wedge}$, we can take conjugation by $\phi_i$ for each $i \not\in \tJ$. By examining their eigenvalues, we have
$$\sum_{\la \in \La ( \vec{m} ), \beta \in \gamma + Q_{\tJ, +}^{\vee}} c_{\la,\beta} [\cO_\beta^\star ( \la ) ] \in \mathbf K_{G} ^{\wedge}.$$
Summing them up with respect to $\vec{m}$, we find that
$$\sum_{\la \in \La, \beta \in \gamma + Q_{\tJ, +}^{\vee}} c_{\la,\beta} [\cO_\beta^\star ( \la ) ] \in \mathbf K_{G} ^{\wedge}.$$

Recall that we have $\bX^{\le}_* \subset \bX_*^{\le} ( \tJ )$ and $\La_+ ( \beta ) \subset \La ^{\tJ}_+ ( \beta ) + \La^{\tI \setminus \tJ}$, and hence there is a natural inclusion between the (labels of the) $\C_q \bX^* ( G )$-basis
\begin{equation}
\{E ( \beta, \la )\}_{\beta \in \bX^{\le}_*, \la \in \La_+ ( \beta )} \subset K_{\bG} ( \Gr_G )\label{Ebasis}
\end{equation}
into the (labels of the) $\C_q \bX^* ( G )$-basis
\begin{equation}
\{E^{\tJ} ( \beta, \la_1 ) e^{\la_2}\}_{\beta \in \bX_*^{\le} ( \tJ), \la _1 \in \La ^{\tJ}_+ ( \beta ), \la_2 \in \La ^{\tI \setminus \tJ}} \subset K_{\bL} ( \Gr _L ).\label{EJbasis}
\end{equation}

If a (formal) linear combination
\begin{equation}
\sum_{\la \in \La, \beta \in \gamma + Q_{+}^{\vee}} c_{\la,\beta} [\cO_\beta^\star ( \la ) ] \hskip 5mm c_{\la, \beta} \in \C_q \bX^* ( G )\label{flG}
\end{equation}
belongs to $K_{\bG} ( \Gr_G )$, then it represents a $\C_q \bX^* ( G )$-linear combination of (\ref{Ebasis}). In view of Proposition \ref{aa-comp}, the partial  sum corresponding to $( \gamma + Q_{\tJ,+}^{\vee} ) \subset ( \gamma + Q_{+}^{\vee} )$ yields the $\C_q \bX^* ( G )$-linear combination of (\ref{EJbasis}) through $K_{\bH} ( \Gr_{H} )$. Therefore, (\ref{flG}) belongs to $K_{\bG} ( \Gr_G )$ only if
$$\sum_{\la \in \La, \beta \in \gamma + Q_{\tJ,+}^{\vee}} c_{\la,\beta} [\cO_{\tJ, \beta}^\star ( \la ) ] \in K_{\bL} ( \Gr _L ).$$
Since the corresponding leading term element belongs to $K_{\bG} ( \Gr_G ) \subset \mathbf K_{G} ^{\wedge}$ as a linear combination of (\ref{Ebasis}) thanks to Lemma \ref{E-basis}, we conclude that $\mathbf K_{G}^{\wedge} \subset \mathbf K_L^{\wedge}$ by removing the leading terms inductively. This forces $K_{\bG} ( \Gr_{G} ) \subset K_{\bL} ( \Gr _L )$ as required. Thus, we conclude Theorem \ref{main}.

\medskip

{\small
\hskip -5.25mm {\bf Acknowledgement:} The author would like thank Michael Finkelberg and \'Eric Vasserot for the discussions on the topic in this paper. He also thanks Joel Kamnitzer for some correspondences. Part of this work was done during his visit to the Higher School of Economics in December 2019, and the Institute Henri Poincar\'e during January--March 2020. This research was supported in part by JSPS KAKENHI Grant Number JP19H01782.}

\appendix

\setcounter{section}{1}
\setcounter{thm}{0}

{\small
\begin{flushleft}
{\normalsize\textbf{Appendix A \hskip 2mm A quantum analogue of the induction equivalence}}
\end{flushleft}

Let $G$ be a connected reductive semi-simple group over $\C$, with a Borel subgroup $B$ and a maximal torus $H$. Let $B \subset P \subset G$ be a parabolic subgroup. Let $Q^{\vee}_{+}$ denote the span of positive coroots (inside the coroot lattice of $G$) identified with the effective cone of $G / B$. Let $Q^{\vee}_{P,+} \subset Q^{\vee}_+$ be the span of positive coroots of $G$ that does not belong to the standard Levi subgroup of $P$ (cf. \S \ref{setup}). Let $\{ (\al^P_i)^{\vee}\}_i$ be the set of positive simple coroots in $Q^{\vee}_{P,+}$.

For a smooth projective variety $\mathfrak X$ over $\C$, we have a subset $H_2 ( \mathfrak X )_+ \subset H_2 ( \mathfrak X, \Z )$ of the effective classes (that is a submonoid). Let $\mathcal M_{g,n,\beta} ( \mathfrak X )$ be the moduli stack of genus $g$ stable maps with $n$-marked points with degree $\beta \in H_2 ( \mathfrak X )_+$ (see \cite{BM96,Lee04}). 

\begin{thm}\label{qK-surj-gen}
Let $X$ be a smooth projective algebraic variety over $\C$ equipped with the $P$-action. We assume $H_1 ( X, \Z ) = \{ 0 \}$. Then, we have a surjective map of algebras
$$QK_G ( G \times_P X ) \longrightarrow QK_P ( X ),$$
where $QK$ denotes the big quantum $K$-group defined in Lee $\cite{Lee04}$.   
\end{thm}

\begin{proof}
Since $X$ is projective with $P$-action, we can consider $X \subset \mathbb P ( V )$ for a finite-dimensional $P$-module $V$. We can twist by $P$-character if necessary to assume that all the $T$-weights $\la$ appearing in $V$ satisfies $\left<( \al^P_i )^{\vee}, \la \right> \ge 0$ for all $i$ (with respect to the standard pairing, cf. \S \ref{setup}). Then, we have an algebraic induction $V^\#$ of $V$, that is the maximal finite-dimensional $G$-module that is generated by $V$. We have $G \times_P X \subset \mathbb P ( V^\# )$, and hence $G \times_P X$ is again projective. The variety $G \times_P X$ is evidently smooth as $X$ is.

Since $H_{1} ( X, \Z ) = 0$, the Leray spectral sequence yields
$$H_2 ( G \times _P X, \Z ) \cong H_2 ( X, \Z ) \oplus H_2 ( G / P, \Z ).$$
The projection map yields
$$\pi : H_2 ( G \times _P X, \Z )_+ \longrightarrow H_2 ( G / P, \Z )_+ \cup \{ 0 \},$$
and the preimage of $0$ is $H_2 ( X, \Z )_+$ by inspection. By the above identification of the effective classes, we find
\begin{equation}
\mathcal M_{g,n,\beta} ( G \times_P X ) \cong G \times_P \mathcal M_{g,n,\beta} ( X )\label{infl-mod}
\end{equation}
whenever $\beta \in \pi^{-1} ( 0 ) \cong H_2 ( X, \Z )_+$. In particular, we have an inflation map
$$\mathrm{infl} : K_P ( \mathcal M_{g,n,\beta} ( X ) ) \stackrel{\cong}{\longrightarrow} K_G ( G \times_P \mathcal M_{g,n,\beta} ( X ) ).$$
By (\ref{infl-mod}), the perfect obstruction theory of $G \times_P \mathcal M_{g,n,\beta} ( X )$ (\cite[\S 2.3 (3)]{Lee04}) can be taken as the inflation of that of $\mathcal M_{g,n,\beta} ( X )$. It follows that
$$\mathrm{infl} ( [\mathcal O^{\mathrm{vir}}_{\mathcal M_{g,n,\beta} ( X )}]) = [\mathcal O^{\mathrm{vir}}_{\mathcal M_{g,n,\beta} ( G \times_P X )}].$$
Note that the quantum $K$-invariants of $\mathcal M_{g,n,\beta} ( X )$ (\cite[\S 4.2]{Lee04}) with respect to the classes from $K_P ( X )$ are $P$-characters (corresponding to finite-dimensional virtual representations of $P$). If $\beta \in \pi^{-1} ( 0 ) \cong H_2 ( X, \Z )_+$, then we find that the inflation isomorphisms $K_P ( X ) \cong K_G ( G \times_P X )$ and $K_P ( \mathcal M_{g,n,\beta} ( X ) ) \cong K_{G} ( \mathcal M_{g,n,\beta} ( G \times_P X ) )$ send the $P$-equivariant Euler-Poincar\'e characteristic maps to $G$-equivariant Euler-Poincar\'e characteristic maps through the algebraic (virtual) induction of the $P$-characters to $G$-characters. In particular, the quantum $K$-potential (\cite[(16)]{Lee04}) of $G \times_P X$ is the inflation of that of $X$ (from $P$-characters to $G$-characters) modulo the Novikov monomial $Q^{\beta}$ with $\pi ( \beta ) \neq 0$. This induces an algebra map
$$QK_G ( G \times_P X ) / ( Q^{\beta} \mid \pi ( \beta ) \neq 0 ) \longrightarrow QK_P ( X )$$
that is an isomorphism as being an isomorphism as vector spaces.
\end{proof}

{\footnotesize
\bibliography{kmref}
\bibliographystyle{hplain}}
\end{document}